\documentclass[twoside]{article}

\usepackage[accepted]{aistats2025}
\usepackage[round]{natbib}

\usepackage{graphicx} 

\usepackage[utf8]{inputenc} 
\usepackage[T1]{fontenc}    
\usepackage{hyperref}       
\usepackage{url}            
\usepackage{booktabs}       
\usepackage{amsfonts}       
\usepackage{nicefrac}       
\usepackage{microtype}      
\usepackage{xcolor}         
\usepackage{natbib}

\usepackage[english]{babel}
\usepackage{amsmath}
\usepackage{graphicx}
\usepackage{amssymb}
\usepackage{amsthm}
\usepackage{mathrsfs}
\usepackage[colorinlistoftodos]{todonotes}
\usepackage{dsfont}
\usepackage{algorithm}
\usepackage[noend]{algpseudocode}
\usepackage{multirow}

\newtheorem{thm}{Theorem}[section]
\newtheorem{lem}[thm]{Lemma}

\newtheorem{defn}[thm]{Definition}

\newtheorem{rmk}[thm]{Remark}
\newtheorem{prop}[thm]{Proposition}

\newcommand{\ie }{\emph{i.e.}, }

\newcommand{\resp}{\emph{resp.} }

\newcommand{\calF}{\mathcal{F}}
\newcommand{\calX}{\mathcal{X}}

\newcommand{\calP}{\mathcal{P}}

\newcommand{\calE}{\mathcal{E}}

\newcommand{\bbR}{\mathbb{R}}

\newcommand{\bbP}{\mathbb{P}}
\newcommand{\bbE}{\mathbb{E}}
\newcommand{\bbS}{\mathbb{S}}

\DeclareMathOperator*{\argmin}{argmin} 

\newcommand{\norm}[1]{\left\lVert#1\right\rVert}
\newcommand{\abs}[1]{\left\vert#1\right\vert}
\newcommand{\floor}[1]{\left\lfloor#1\right\rfloor}

\newcommand{\indicator}{\mathds{1}}
\newcommand{\iid}{\overset{i.i.d}{\sim}}
\newcommand{\set}[1]{\left\{#1\right\}}
\DeclareMathOperator*{\tr}{tr}

\usepackage[inline]{trackchanges}
\addeditor{jp}
\addeditor{jk}
\addeditor{ab}

\usepackage{tikz}
\def\checkmark{\tikz\fill[scale=0.4](0,.35) -- (.25,0) -- (1,.7) -- (.25,.15) -- cycle;}
\usetikzlibrary{angles,
                quotes}
\usepackage{siunitx}

\usepackage{paralist}

\begin{document}

%

%

\twocolumn[

\aistatstitle{Robust Estimation in metric spaces: Achieving Exponential Concentration with a Fr\'echet Median}

\aistatsauthor{Jakwang Kim* \And Jiyoung Park* \And Anirban Bhattacharya}

\aistatsaddress{University of British Columbia \And  Texas A\&M University \And Texas A\&M University}]

\begin{abstract}
There is growing interest in developing statistical estimators that achieve exponential concentration around a population target even when the data distribution has heavier than exponential tails. More recent activity has focused on extending such ideas beyond Euclidean spaces to Hilbert spaces and Riemannian manifolds. In this work, we show that such exponential concentration in presence of heavy tails can be achieved over a broader class of parameter spaces called CAT($\kappa$) spaces, a very general metric space equipped with the minimal essential geometric structure for our purpose, while being sufficiently broad to encompass most typical examples encountered in statistics and machine learning. The key technique is to develop and exploit a general concentration bound for the Fr\'echet median in CAT($\kappa$) spaces. We illustrate our theory through a number of examples, and provide empirical support through simulation studies.
\end{abstract}

\section{Introduction}

A fundamental challenge in statistical estimation pertains to dealing with heavy-tailed data. Many estimators used in practice are built upon the assumption of light-tailed data, and their finite-sample (or non-asymptotic) statistical properties do not necessarily carry over when the data generating distribution has heavy tails. A standard way to characterize such non-asymptotic behavior of statistical estimators is via concentration inequalities; see \cite{boucheron2013concentration} for a book-level treatment. Given an estimator $\widehat{\theta}_n$ based on $n$ samples for parameter $\theta$, a concentration inequality provides a \emph{non-asymptotic bound} to some distance $d(\widehat{\theta}_n, \theta)$ being greater than a tolerance level $\varepsilon > 0$, as a function of $n$ and $\varepsilon$. Inverting such a concentration inequality, one can obtain a growth rate on the sample size as a function of the tolerance level $\varepsilon$ to provably achieve a desired level of confidence (say $95\%$). As a simple illustrative example, the sample average of i.i.d. real-valued observations concentrate exponentially fast (in terms of sample size) around the population mean if the true data distribution has sub-exponential tails by Bernstein's inequality. However, if one weakens the tail assumption to expand the class of true distributions to those with finite second moment, then one can only establish polynomial concentration as dictated by Chebyshev's inequality \citep{catoni2012challenging}. Exponential concentration is desirable not only to obtain logarithmic dependence of the sample size on the tolerance level for a single estimator, but also to combine multiple dependent estimators via a union bound. Accordingly, there has been extensive recent research towards constructing alternative `robust' estimators that achieve exponential concentration rates in situations where standard M-estimators (or method of moment estimators) fail to provide one \citep{minsker2019uniformboundsrobustmean, jean2011robust, oliveira2024improvedcovarianceestimationoptimal}.





Modern statistics and machine learning routinely encounter data beyond the classical Euclidean settings. Hyperspheres are used to model the directional data and spatial data \citep{hall1987kernel, jeong2017spherical, zhang2021kernel}; Hilbert spaces serve as a base space in functional data analysis \citep{petersen2016functional}; hyperbolic spaces have become popular for hierarchical data \citep{nickel2017poincare}; and graph and tree data are predominant in network data analysis \citep{fortunato2010community, abuata2016metrictree}. Accordingly, several recent attempts have been made to address statistical problems in more general spaces; see, e.g., \citet{arnaudon2013medians, brunel2023concentration, holmes2003trees, kostenberger2024robust, romon2023convex, sturm2000NPC} for some representative examples.


Motivated by the extensive literature on robust estimations and statistical methods beyond Euclidean spaces, this work proposes a method for robust estimation--boosting weakly concentrating estimators to strongly concentrate--in general metric spaces under minimal assumptions.
Specifically, we focus on CAT($\kappa$) spaces, which is a general metric space with the minimal essential geometric structure necessary for our purposes, and introduce a procedure to boost weakly concentrating estimators in these spaces. CAT($\kappa$) spaces encompass not only widely studied spaces such as Hilbert spaces (Euclidean spaces) and Riemannian manifolds, but also other various spaces that have gained attention in data science and yet have been less examined in the context of robust estimation; see Section \ref{section_prelim_catk} for examples.


The contributions of this work are as follows:
\begin{compactitem}
    \item[(1)] We propose a method to boost estimators with polynomial concentration to strong (exponential) concentration in CAT($\kappa$) spaces; a 
    general setting that encompasses most spaces concerned in statistics and machine learning; by leveraging 
    properties of the Fréchet median (Section \ref{section_robustify}).
    \item[(2)] We show the applicability of our method across a wide range of statistical problems (Section \ref{section_applications}).
    \item[(3)] We show the proposed method allows for a tractable algorithm, and verify its strength numerically (Section \ref{section_implementation_experi}).
\end{compactitem}

\section{Preliminaries}\label{section_prelim}

\subsection{CAT($\kappa$) space}\label{section_prelim_catk}

In this section, we introduce CAT($\kappa$) spaces. A more rigorous definition is provided in Appendix \ref{appendix_alexandrov_catk}.

For each $\kappa \in \mathbb{R}$, the 2-dimensional model spaces $M^2_\kappa$ are defined as follows:
\[
M^2_\kappa = \begin{cases}
    \mathbb{R}^2 &\quad \text{if $\kappa=0$},\\
    \frac{1}{\sqrt{\kappa}}\mathbb{S}^2 &\quad \text{if $\kappa>0$},\\
    \frac{1}{\sqrt{\abs{\kappa}}}\mathbb{H}^2 &\quad \text{if $\kappa < 0$}
\end{cases}
\]
where $\mathbb{S}^2$ and $\mathbb{H}^2$ are the 2-dimensional unit sphere and hyperbolic plane, respectively. A geodesic space $(\calX, d)$ is a CAT($\kappa$) space if every geodesic triangle in $\calX$ has a \emph{comparison triangle} in $M^2_\kappa$ with the same side lengths, such that the original triangle is \emph{thinner} (a precise  mathematical formulation is given in Appendix \ref{appendix_alexandrov_catk}). This condition, known as the \emph{CAT($\kappa$) inequality}, describes how much the space is curved, allowing curvature of spaces to be defined with minimal structure. See Figure \ref{figure:catk_triangle}. By definition, $\mathbb{R}^2$, $\mathbb{S}^2$, and $\mathbb{H}^2$ are CAT(0), CAT(1), and CAT(-1) spaces, respectively.

\begin{figure}[t]
\centering
\includegraphics[width=0.16\textwidth]{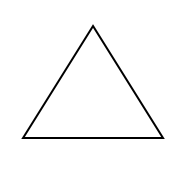}\hfill
\includegraphics[width=0.16\textwidth]{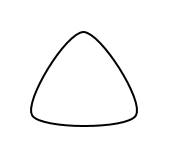}\hfill
\includegraphics[width=0.16\textwidth]{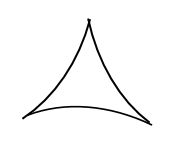}
%
\caption{Triangles in $M_{\kappa}^2$ spaces for different $\kappa$. Triangles in a CAT($\kappa$) space is thinner than triangles in $M_{\kappa}^2$. \textbf{Left}: Euclidean triangle ($\kappa = 0$). \textbf{Middle}: Spherical triangle ($\kappa > 0$). \textbf{Right}: Hyperbolic triangle ($\kappa < 0$).  }\label{figure:catk_triangle}
\end{figure}



A lot of spaces studied in practice belong to the category of CAT($\kappa$) spaces. Below are some examples. 
\begin{compactitem}
    \item Riemannian manifolds with sectional curvature upper bounded by $\kappa$: Hyperspheres, hyperbolic spaces, information geometry, Kendall shape spaces, space of Symmetric Positive Definite (SPD) matrices, to name a few.
    \item Infinite dimensional spaces: Hilbert spaces and infinite dimensional hyperbolic spaces are prominent examples of CAT(0). 
    \item Singular spaces and stratified spaces: These spaces have gained interest recently \citep{geiger2001stratifiedexponential, mattingly2023centrallimittheoremsfrechet, mattingly2023geometrymeasuressmoothlystratified, mattingly2023shadowgeometrysingularpoints} but do not fall under classical spaces. On the other hand, many of these spaces are CAT($\kappa$) spaces \citep{burago2001course}[Theorem 9.1.21]. 
    \item Spaces of phylogenetic trees: Phylogenetic trees are widely studied objects in the field of biology and statistics. The space of phylogenetic trees can be endowed with a metric that makes it CAT(0) \citep{billera2001geometry}.
    \item Metric graphs and trees: Any metric graph with cycles of length less than $2 \pi$ is CAT(1) \citep{brown2016gluing}[Remark 2.12]. Specifically, metric trees are CAT(0).
\end{compactitem}

In particular, CAT(0) spaces are often referred to as Hadamard spaces or Non-Positively Curved (NPC) spaces in the literature. They have drawn attention in statistics community due to their applicability to practical examples and favorable properties when incorporating a probability measure \citep{arnaudon2013medians, brunel2023concentration, LeGouic2019FastCO, kostenberger2024robust, romon2023convex, sturm2000NPC, yun2023expoential}. These favorable properties are discussed in Remark \ref{rmk_frechet_exist_unique}.

\subsection{Fr\'echet mean and median}\label{section_frechet_mean_median}

In this section, we introduce the notions of Fr\'echet mean and median, which are generalizations of classical mean and median to a general separable and complete metric space $(\calX, d)$. 
Define $\calP_p(\calX)$ as the set of Borel probability measures $P$ on $\calX$ with a finite $p^{th}$ moment, \ie $\int_{\calX} d^p(x,y) dP(y) < \infty$ for some $x \in \calX$. 


\begin{defn}[Fr\'echet mean and median]
    Given $P \in \calP_p(\calX)$, suppose $x^* \in \calX$ satisfies
    \[
    x^* \in \argmin_{x \in \calX} \int_{\calX} d^p(x,y) dP(y).
    \]
    Any such $x^*$ with $p = 1$ is called a Fr\'echet median of $P$, and with $p=2$ is called a Fr\'echet mean of $P$.
\end{defn}

Fr\'echet mean and median are also referred to as \emph{barycenter} and \emph{geometric median}, respectively, in some literature. 

\begin{rmk}[Existence and uniqueness of the Fr\'echet mean and median]\label{rmk_frechet_exist_unique}
For NPC spaces, Fr\'echet mean (\resp median) globally exists for any probability measure with a finite second (\resp first) moment \citep{bacak2014computing}[Lemma 2.3]. In addition, Fr\'echet mean is unique \citep{bacak2014computing}[Theorem 2.4]. Fr\'echet median may not be unique, but when it is not unique they form a single geodesic segment in the same manner medians behave in $\bbR$ \citep{schötz2024varianceinequalitiestransformedfrechet}[Theorem 6.6]. On the other hand, for CAT($\kappa$) with $\kappa > 0$, a probability measure in addition needs to be supported within a ball of radius smaller than $D_{\kappa}/2$ to guarantee the (unique) existence of a Fr\'echet median (resp. mean) \citep{yokota2017convexfunctions}. 
\end{rmk}

We lastly note that the Fr\'echet mean and median are not the only extensions of the classical mean and median to a general metric space. \citet{sturm2000NPC}[Section 7] introduced a convex mean that coincides with the Fr\'echet mean in Euclidean space but not necessarily in a general metric space. Similarly, \citet{lugosi2016risk} suggested a tournament-based median, which matches the Fr\'echet median in $\bbR$ but not beyond it. In addition, \citet{Dai2021TukeysDF} suggested to extend the idea of Tukey's depth based median to general geodesic spaces.

\subsection{Robust estimation and a median-of-means estimator}\label{section_mom}


\citet{catoni2012challenging} was one of the first to demonstrate the existence of an M-estimator that achieves exponential concentration only under moment conditions in $\bbR$. Since then, numerous studies have explored the existence of estimators with similar favorable properties in more general spaces. For instance, \citet{lugosi2021robust} generalized the notion of `trimmed mean' to $\bbR^d$ and showed its exponential concentration.


One of the most extensively studied methods for this pursuit is a median-of-means (MoM) estimator \citep{nemirovsky1983problemcomplexity, alon1996thespacecomplexity}. The construction of a MoM estimator is as follows:  first split $n$ data into $k$ disjoint blocks, compute sample means for each block, and then take a median over those means. A MoM estimator interpolates between the mean estimator (when $k = 1$) and the median estimator (when $k = n$). Therefore, choosing an appropriate $k$ guarantees accurate estimation of the mean with a level of robustness similar to that of the median. 
\citet{lerasle2011robust} showed MoM estimators achieve exponential concentration over the real line. Beyond $\bbR$, different choices of medians lead to different MoM estimators. \citet{minsker2015geometric, lin2020RobustOA} utilize the Fr\'echet median as a median and demonstrate the existence of MoM estimators in Banach spaces and certain families of Riemannian manifolds. Meanwhile, \citet{lugosi2019subgaussian, yun2023expoential} define MoM via a ``median-of-means tournament'' and derive concentration bounds in $\bbR^d$ and NPC spaces, although their MoM estimators face computational issues \citep{lugosi2019subgaussian, yun2023expoential}[Section 4, 6]. The philosophy of MoM has been extended to more general target quantities beyond the mean. For example, \citet{minsker2017subsetposterior} proposed a robust posterior distribution using MoM-inspired ideas. 

\section{Main result: Boosting the weak estimators by Fr\'echet median}\label{section_robustify}

In this section, we adopt the Fr\'echet median robust estimation in CAT($\kappa$) spaces based on the method proposed by \citet{minsker2015geometric}. \citet{minsker2015geometric}[Lemma 2.1] establishes a key link between the Fr\'echet median and robustness in Hilbert spaces, which is later extended to Riemannian manifolds with some assumptions by \citet{lin2020RobustOA}[Lemma 2.1].
The idea is that the empirical Fr\'echet median controls the geometric discrepancies between data points, which guarantees the desirable concentration. One of the main contributions of this work is the extension of this idea to general CAT($\kappa$) spaces.

For $(\calX, d)$ a CAT($\kappa$) space, an empirical Fr\'echet median of $x_1, \dots, x_k \in \calX$ is defined by the Fr\'echet median of the empirical measure $\sum_{j=1}^{k}\delta_{x_j}/k$. We will use the notation $\text{med}\left(x_1, \dots, x_k \right)$ to denote the empirical Fr\'echet median of $x_1, \dots, x_k$.



\begin{lem}[Geometric discrepency near the Fr\'echet median] \label{lem_geo_med}
    Let $(\calX, d)$ be a CAT($\kappa$) space, and fix $x_1, \dots, x_k \in \calX$. Denote $x^* := \text{med}\left(x_1, \dots, x_k\right)$. Fix $\alpha \in (0,0.5)$ and write $C_{\alpha} = (1-\alpha)(1-2\alpha)^{-1/2}$. Suppose either (a) or (b) holds:
    \begin{compactitem}
        \item[(a)] For $\kappa \leq 0$, assume there exists $z \in \calX$ such that $d(x^*, z) > C_{\alpha} r$ for some $ r > 0$.
        \item[(b)] For $\kappa > 0$, write $D_{\kappa} = \pi/\sqrt{\kappa}$. Assume $x^*$ exists, $x_j \in B(x^*, D_{\kappa}/2)$ for all $j = 1, \ldots, k$, and there exists $z \in \calX$ such that $ \frac{\pi}{2} C_{\alpha} r < d(x^*, z) \leq D_{\kappa}/2$ for some $0 < r < D_{\kappa}/(C_{\alpha} \pi)$.
    \end{compactitem} 
Under (a) or (b), there exists a subset $J \subseteq \{1, \dots, k\}$, with cardinality $|J| > \alpha k$, such that for all $j \in J$, $d(x_j, z) > r$.
\end{lem}

Lemma \ref{lem_geo_med} implies that if a point $z$ is far away from a Fr\'echet median, then it also has to be far away from the bulk of the points, $x_j$'s. The proof of Lemma \ref{lem_geo_med} relies on the behavior of the triangle $\triangle x_j x^* z$ in CAT($\kappa$) spaces. Utilizing the CAT($\kappa$) inequality, or the triangle comparison, it turns out that such $z$ cannot be close to $x_j$'s. A complete proof will be provided in Appendix \ref{appendix_proof_section_robustify}.

\begin{rmk}\label{rmk_curvature}
    \begin{compactitem}
        \item []
        \item[1.] 
        When $\kappa > 0$, the assumptions on the positions of points and the existence of $x^*$ are required. As mentioned in Remark \ref{rmk_frechet_exist_unique}, this is possible whenever points are distributed in a ball of radius smaller than $D_{\kappa}/2$.
        \item[2.] Lemma \ref{lem_geo_med} covers the case of Hilbert spaces discussed in \citet{minsker2015geometric}, but does not extend to Banach spaces. This limitation arises from the fact that, unlike Hilbert triangles, Banach triangles may not satisfy inner-product-based inequalities such as the Cauchy-Schwarz inequality. For more details, see \citet{khamsi2017generalizedcat0}.
        \item[3.] Restricted to Riemannian manifolds, the curvature upper bound condition in Lemma \ref{lem_geo_med} implies the Lipschitz logarithmic map condition proposed by \citet{lin2020RobustOA}[Lemma 2.1], with the Lipschitz constant being $\pi/2$. We conjecture that the converse is also true: the Lipschitz logarithmic map condition is valid only if the sectional curvature at $x^*$ is upper bounded. If this conjecture holds, then Lemma \ref{lem_geo_med} precisely encompasses the findings of \citet{lin2020RobustOA}. 
        
    \end{compactitem}
\end{rmk}

Now, we proceed to the main theorem.

\begin{thm}[Boosting a weak estimator]\label{thm_concen_of_med}
    Suppose the parameter space $\Theta$ is CAT($\kappa$) space. Let $\theta \in \Theta$ be a parameter of interest and $\widehat{\theta}_{j}$, $j = 1, \dots, k$ be independent estimators of $\theta$. Let $\widehat{\theta}_{FMoE}:= \text{med}\left(\widehat{\theta}_1, \dots, \widehat{\theta}_k\right)$ be a `Fr\'echet median of estimators'.
    
    Fix $\alpha \in (0,1/2)$ and $p \in (0, \alpha)$. Write $\psi(\alpha, p) := (1-\alpha)\log \frac{1-\alpha}{1-p} + \alpha \log \frac{\alpha}{p}$ and set $C_{\alpha}$ same as in Lemma~\ref{lem_geo_med}.
    \begin{compactitem}
        \item[(a)] For $\kappa \leq 0$, suppose there exists $\epsilon > 0$ such that $\bbP\left(d(\widehat{\theta}_j, \theta) > \epsilon\right) \leq p$ for all $j = 1, \dots, k$. Then,
        \begin{align*}
            &\bbP\left[d(\widehat{\theta}_{FMoE}, \theta) > C_{\alpha} \epsilon\right] \leq \exp\left(-k\psi(\alpha, p)\right).
        \end{align*}
    \item[(b)] For $\kappa > 0$, suppose there exists $\epsilon \in (0, D_{\kappa}/(\pi C_{\alpha}))$ such that $\bbP\left(d(\widehat{\theta}_j, \theta) > \epsilon\right) \leq p$ for all $j = 1, \dots, k$. Assume $\widehat{\theta}_{FMoE}$ exists, $\widehat{\theta}_j \in B(\widehat{\theta}_{FMoE}, D_{\kappa}/2)$, and  $\widehat{\theta}_{FMoE} \in B(\theta, D_{\kappa}/2)$ almost surely. Then,  
    \begin{align*}
        &\bbP\left[d(\widehat{\theta}_{FMoE}, \theta) > \frac{\pi C_{\alpha} \epsilon}{2} \right] \leq \exp\left(-k\psi(\alpha, p)\right).
    \end{align*}
    \end{compactitem}
\end{thm}

We provide the proof in Appendix \ref{appendix_proof_section_robustify}.

\begin{rmk}[$\kappa > 0$ case]\label{remark_boosting_k_positive}
For $\kappa > 0$, a priori concentrations $d(\widehat{\theta}_{FMoE}, \widehat{\theta}_j) \leq D_{\kappa}/2$ and $d(\widehat{\theta}_{FMoE}, \theta) \leq D_{\kappa}/2$ are required. 
The first condition is a common assumption in CAT($\kappa$) spaces \citep{brunel2023concentration}[Theorem 18]. The second condition is new here. 
If one avoids using this condition, one only obtains a conditional form of concentration: if $\widehat{\theta}_j \in B(\widehat{\theta}_{FMoE}, D_{\kappa}/2 - \epsilon)$ almost surely, one can obtain
\begin{align*}
        &\bbP\left[d(\widehat{\theta}_{FMoE}, \theta) > \frac{\pi C_{\alpha} \epsilon}{2} \bigg\vert d(\widehat{\theta}_{FMoE}, \theta) \leq D_{\kappa}/2\right] \\
        &\leq \frac{\exp\left(-k\psi(\alpha, p)\right)}{1-p^k}.
\end{align*}
\end{rmk}
When the context is clear, we will refer to $\widehat{\theta}$ as the `original estimator' throughout the paper. Theorem~\ref{thm_concen_of_med} states that by taking the Fr\'echet median, the original estimator can be boosted to achieve exponential concentration even when only weak, e.g., polynomial, concentration would be expected. Algorithm \ref{alg_moe} outlines the procedure for constructing a Fr\'echet median of estimators (FMoE) based on the original estimator.

\begin{algorithm}
\begin{algorithmic}[1]
\caption{Boosting a preliminary estimator.}\label{alg_moe}
\Require Input data $x_1, \dots, x_{n}$, CAT($\kappa$) space $(\Theta, d)$, the block size $k$.

\State Split the data $x_i$ into $k$ disjoint blocks, with each block consisting of $\floor{n/k}$ data points.
\For{$j \gets 1$ to $k$}                   
    \State $\widehat{\theta}_j \gets$ The original estimator from $j^{th}$ block data points.
\EndFor
\State \Return $\widehat{\theta}_{FMoE} \gets$ Frechet median of $\widehat{\theta}_j$ with respect to the metric $d_{\Theta}$.
\end{algorithmic}
\end{algorithm}

\begin{rmk}[Time complexity of Algorithm~\ref{alg_moe}]
To the best of our knowledge, the time complexity of Algorithm~\ref{alg_moe} in general setting is unknown, as that of computing the Fr\'echet median is unknown beyond Euclidean spaces. That said, 
for the time complexity with respect to $n$ for a
fixed $d$, one can expect the time complexity of FMoE matches 
the time complexity of the original estimator. A heuristic argument proceeds as follows: suppose the time complexities of an original estimator and Fr\'echet median given $n$ elements are $O(n^{\alpha})$ and $O(n^{\beta})$ for some $\alpha, \beta > 0$. Then, given $k$, the time complexity of obtaining the original estimators over $k$ blocks will be $O(n^{\alpha} k^{1-\alpha})$. In sum, the total time complexity will be $O(n^{\alpha} k^{1-\alpha} + k^{\beta})$. Typically, both from theoretical observations and numerical simulations, $k$ is relatively small compared to $n$ (see Sections \ref{section_applications} and \ref{section_implementation_experi}). Therefore, if one assumes $k = O(1)$, $O(n^{\alpha} k^{1-\alpha})$ term dominates as $n$ gets larger, meaning the complexity matches the original estimator's time complexity $O(n^{\alpha})$.
\end{rmk}

We conclude this section by highlighting the favorable properties of our method. First, it can boost any type of estimator. While most research on statistical estimation in general metric spaces has focused on Fr\'echet mean estimation \citep{ahidarcoutrix2018ConvergenceRF, brunel2023concentration, LeGouic2019FastCO, sturm2000NPC, yun2023expoential}, our method also applies to estimators that need not be Fr\'echet means. We illustrate the advantages of this broad applicability in Section \ref{section_applications_beyondmean}.

Second, our method is nearly fully implementable. Given the original estimators, Algorithm \ref{alg_moe} requires only the computation of the Fréchet median of them. This is always feasible in NPC spaces \citep{bacak2014computing}[Algorithm 4.3]. In CAT($\kappa$) spaces with $\kappa > 0$, while no universal algorithm for Fr\'echet median exists, algorithms tailored to specific domains can be utilized \citep{fletcher2009geometricmedian, boria2019generalizedmedian}. This improves upon \citet{yun2023expoential}, which considered the median-of-means estimators in NPC spaces using a tournament-based median but lacked computational tractability.

\section{Statistical applications}\label{section_applications}

This section illustrates how the proposed method boosts some widely used estimators to achieve exponential concentration without sub-Gaussian conditions on the data. While our method is generally applicable, we investigate the boosting of two important estimators as examples: (1) Fr\'echet mean estimators in NPC spaces and (2) the canonical sample covariance estimator.



\subsection{Boosting Fr\'echet means: Fr\'echet median-of-means estimators}\label{section_applications_mom}

As the primary application, we focus on the widely studied problem of estimating the Fr\'echet mean in NPC spaces \citep{brunel2023concentration, LeGouic2019FastCO, sturm2000NPC, yun2023expoential}. This problem has received significant attention in recent years due to the existence guarantee of the Fr\'echet mean and median in these spaces, as discussed in Remark \ref{rmk_frechet_exist_unique}. For this problem, the proposed method becomes a Fr\'echet median-of-means (FMoM). Throughout this section we assume that $(\calX, d)$ is a NPC space with a curvature lower bound $\kappa_{\min}(\calX) \in [-\infty, 0]$ and that $x_i \iid P \in \calP_2(\calX)$ for $i =1, \dots, n$. We denote the Fr\'echet mean and the second moment of $P$ as $x^*$ and $\sigma^2 := \bbE_{y \sim P}\left[d^2(x^*, y)\right]$ respectively.

There are two natural ways to construct a Fr\'echet mean estimator in NPC spaces: (1) Empirical Fr\'echet mean $\widehat{x}_{EM}$ and (2) inductive mean $\widehat{x}_{IM}$. These serve as alternative estimators of the population mean though they differ in some aspects. The explanations of these estimators will be provided in Appendix \ref{appendix_frechet_mean_estimation}. Proposition \ref{prop_expected_error_mean_estimators} shows a weak concentration of these estimators under mild conditions. 

\begin{prop} \label{prop_expected_error_mean_estimators}
    Let $\widehat{x}$ be either $\widehat{x}_{EM}$ or $\widehat{x}_{IM}$. If $\widehat{x} = \widehat{x}_{EM}$, in addition assume $\kappa_{\min}(\calX) > -\infty$. Then,
    \[
    \bbE\left[d^2(\widehat{x}, x^*)\right] \leq \frac{\sigma^2}{n}.
    \]
    Furthermore, for any $\epsilon > 0$
    \[
        \mathbb{P}\left[ d(\widehat{x}, x^*) > \epsilon \right] \leq \frac{\sigma^2}{n \epsilon^2}.
    \]
\end{prop}
\begin{proof}
    For the expected error bound, see \citet{LeGouic2019FastCO}[Corollary 3.4] for $\widehat{x}_{EM}$ case and \citet{sturm2000NPC}[Theorem 4.7] for $\widehat{x}_{IM}$ case.

    The concentration inequality directly follows from Markov inequality.
\end{proof}

On the contrary, achieving the exponetial concentration requires sub-Gaussian type assumptions, which are stronger than the usual sub-Gaussian conditions in Euclidean space; see the discussion in \citet{brunel2023concentration}[Definition 3].

\begin{prop}\citet{brunel2023concentration}[Corollary 11]\label{prop_exp_tail_bound_subgaussian}
    Suppose $P$ satisfies the following sub-Gaussian type assumption, \ie 
    \[
    \sup_{f \in \calF} \bbE_{X \sim P}\left[e^{\lambda(f(X) - \bbE[f(X)])}\right] \leq e^{\frac{\lambda^2 K^2}{2}} \qquad \forall \lambda > 0
    \]
    where $\calF = \{f: \calX \rightarrow \bbR \: \vert \: f \text{ is a $1$-Lipschitz function}\}$. Let $\widehat{x}$ be either $\widehat{x}_{EM}$ or $\widehat{x}_{IM}$. If $\widehat{x} = \widehat{x}_{EM}$, in addition assume $\kappa_{\min}(\calX) > -\infty$. Then,
    \[
    \bbP\left[d(\widehat{x}, x^*) \geq \frac{\sigma}{\sqrt{n}} + K \sqrt{\frac{\log(1/\delta)}{n}}\right] \leq \delta.
    \]
\end{prop} 

Now, we turn our attention to FMoM. Theorem \ref{thm_empirical_mean_mom} shows a FMoM achieves the exponential concentration without the sub-Gaussian assumptions, just as in Hilbert space.

\begin{thm}\label{thm_empirical_mean_mom}
    Fix $\delta > 0$. Let $\widehat{x}_{FMoM}$ be a Fr\'echet median-of-means of $\widehat{x}_j$'s where $\widehat{x}$ is either $\widehat{x}_{EM}$ or $\widehat{x}_{IM}$. Set $k = \floor{\log(1/\delta) / \psi(7/18, 1/10)} + 1$.  If $\widehat{x} = \widehat{x}_{EM}$, in addition assume $\kappa_{\min}(\calX) > -\infty$. Then
        \[
        \bbP\left[d(\widehat{x}_{FMoM}, x^*) \geq 11\sqrt{\frac{\sigma^2 \log(1.4/\delta)}{n}} \right] \leq \delta.
        \]
\end{thm}

The proof is provided in Appendix \ref{appendix_proof_section_applications}.


\begin{rmk}[Fr\'echet mean estimation for $\kappa > 0$]
In the Fr\'echet mean estimation problem in CAT($\kappa$) with $\kappa > 0$, the bounded support condition, as mentioned in Remark \ref{rmk_frechet_exist_unique}, is necessary to ensure the unique existence of the Fr\'echet mean. One might wonder whether this assumption automatically implies exponential concentration of the empirical Fr\'echet mean, as in Euclidean space. However, even in this case, proper finiteness of the metric entropy must be imposed to guarantee exponential concentration \citep{brunel2023concentration}[Theorem 18]. Conversely, in spaces with $\kappa_{\min} \geq 0$ that satisfy the so-called extendible geodesics condition, polynomial concentration is achieved \citep{LeGouic2019FastCO}[Theorem 3.3, 4.2]. Notably, this condition can be met even if the space lacks a strong metric entropy bound \citep{ahidarcoutrix2018ConvergenceRF, LeGouic2019FastCO}[Example 2.6, Corollary 4.4].
\end{rmk}


\subsection{Boosting a sample covariance estimator}\label{section_applications_beyondmean}

As mentioned in Section \ref{section_robustify}, the main advantage of this work lies in its extendability to problems beyond the Fr\'echet mean estimation. The proposed method is applicable for inducing a robust estimator whenever (1) a parameter space can endow a CAT($\kappa$) structure, and (2) the original estimator achieves weak concentration (e.g. polynomial). In this section, its application to the estimation of the sample covariance matrix is provided.

Unlike the Fr\'echet mean problem discussed in Section~\ref{section_applications_mom}, original estimators may not necessarily rely on the geometry of CAT($\kappa$) space. The main difficulty in this section lies in obtaining weak concentration of the original estimators with respect to the metric of a CAT($\kappa$) space. Fortunately, an estimator built under Euclidean geometry achieves the polynomial concentration with respect to the metrics of a CAT($\kappa$) space for this sample covariance problem. Once polynomial concentration for the original estimator is established, the procedure goes similarly to Theorem \ref{thm_empirical_mean_mom}, resulting in exponential concentration for a FMoE estimator. 





\subsubsection{Geometry of symmetric positive definite matrices}

Symmetric positive definite (SPD) matrices arise in many fields, hence their estimation and concentration are an important issue. While there are several geometries in matrix spaces, two metrics are specifically tailored to SPD matrices: affine invariant metric and Bures-Wasserstein metric.

First, the affine invariant metric is defined as follows:
\[
    d_{AI}(A,B) := \|\log A^{-1/2} B A^{-1/2}\|_{F}.
\]
This metric coincides with the Fisher-Rao metric between multivariate Gaussian distributions with fixed mean and covariance matrices $A$ and $B$ \citep{nielsen2023asimpleapproximation}. Additionally, the Fr\'echet mean of SPD matrices with respect to $d_{AI}$ coincides with the geometric mean and plays an important role in diffusion tensor imaging \citep{fillard2005extrapolation}. 

The Bures-Wasserstein metric is defined as follows:
\[
d_{BW}^2(A,B) := \tr(A) + \tr(B) -2 \tr(A^{1/2}BA^{1/2})^{1/2}.
\]
This metric arises naturally in the fields of quantum information and optimal transport. Particularly, this metric is a Wasserstein distance between two multivariate Gaussian distributions with fixed mean and covariance matrices A and B \citep{bhatia2019bureswasserstein}.

While these two metrics both inherit the SPD constraint naturally, their geometries are quite different. The metric space $(SPD, d_{AI})$ forms a NPC space \citep{bhatia2006riemannian}[Proposition 5]. In contrast, $(SPD, d_{BW})$ forms a non-negative curvature spaces and is not in fact a CAT($\kappa$) space for any $\kappa > 0$ \citep{takatsu2009wassersteingeometryspacegaussian}[Theorem 1.1]. However, when restricted to the set of SPD matrices whose smallest eigenvalue is lower bounded by $\sqrt{3/(2\kappa)}$, it becomes a CAT($\kappa$) space \citep{massart2019curvature}[Proposition 2]. 


\subsubsection{Concentration analysis}

One difficulty in the sample covariance matrix estimation problem lies in the SPD matrix constraint. Many analyses of covariance matrices utilizing matrix norms may potentially violate this constraint. For example, it is non-trivial whether taking a matrix norm Fr\'echet median of SPD matrices satisfies the SPD matrix constraint.

Conversely, since $(SPD, d_{AI})$ is a NPC space, the Fr\'echet median of SPD matrices with respect to $d_{AI}$ resides within the SPD space as discussed in Section \ref{section_frechet_mean_median}. For $(SPD, d_{BW})$ one requires additional assumptions that the eigenvalues of matrices should be lower bound by $\sqrt{3/(2\kappa)}$ and matrices are supported within a ball of radius smaller than $D_{\kappa}/2$. Under these additional conditions, one can also guarantee that their Fr\'echet median with respect to $d_{BW}$ is SPD, maintaining the eigenvalue lower bound. In this regard, our method produces an estimator that concentrates exponentially with respect to the chosen metric while satisfying the SPD constraint. We set the original estimator as the canonical sample covariance matrix $\widehat{\Sigma} = \sum_{i=1}^{n} X_i X_i^{T}/n$ (assuming the mean being 0 for simplicity). However, as noted in Section \ref{section_robustify}, our approach is applicable to any covariance estimator that exhibits weak concentration with respect to the chosen metric; e.g., the Fréchet mean of $X_i X_i^T$'s, which corresponds to FMoM discussed in Section \ref{section_applications_mom}.


The following proposition shows the weak concentration of sample covariance matrix estimator w.r.t. both $d_{AI}$ and $d_{BW}$.

\begin{prop}[Polynomial tail bound for covariance matrix estimator]\label{prop_tail_of_cov_matrix}
    Let $X_i \iid P \in \calP_4(\bbR^d)$ a distribution with mean 0 and covariance matrix $\Sigma$ with a fixed dimension $d$. Let $\widehat{\Sigma} = \sum_i X_i X_i^T / n$ be a sample covariance estimator. Then, writing $\lambda_{\min}$ the smallest eigenvalue of $\Sigma$,
    \[
    \bbP\left[d_{AI}\left(\widehat{\Sigma}, \Sigma \right) \geq \epsilon \right] \leq \frac{C d^4}{n \lambda_{\min}^2\left(1 - \exp\left(-\frac{\epsilon}{\sqrt{d}}\right)\right)^2}
    \]
    for some constant $C > 0$ only depends on the moments of $P$.

    For $d_{BW}$, in addition assume both $\widehat{\Sigma}$ and $\Sigma$ have the eigenvalue lower bound by $\lambda_0 > 0$. Then
    \[
    \bbP\left[d_{BW}\left(\widehat{\Sigma}, \Sigma \right) \geq \epsilon \right] \leq \frac{C d^4}{4 n \lambda_0 \epsilon^2}
    \]
    with the same $C$ in the above.
\end{prop}
We provide a proof in Appendix \ref{appendix_proof_section_applications}.

\begin{rmk}\label{rmk_eigenvalue_lower_bound_concentration}
    \begin{enumerate}
        \item []
        \item The above bound may not be optimal with respect to the dimension $d$. We do not pursue obtaining the optimal dimension bound, as our main goal is to verify the polynomial concentration with respect to $n$.
        \item For the $d_{BW}$ bound, the eigenvalue lower bound assumption on $\widehat{\Sigma}$ may not sound plausible at glance, as $\lambda_{\min}(\widehat{\Sigma})$ is a random quantity. However, its value will be in fact concentrated in $B(\lambda_{\min}(\Sigma), C/\sqrt{n})$ for some universal constant $C > 0$ whenever $P$ has a finite fourth moment; see Appendix \ref{appendix_proof_section_applications}. Accordingly, it is not too harmful to regard $\lambda_0 = \lambda_{\min}(\Sigma)/2$ in practical applications.
    \end{enumerate}
\end{rmk}

Given weak concentration of $\widehat{\Sigma}$, we can proceed with the same technique in Theorem \ref{thm_empirical_mean_mom} to obtain the exponential concentration rate of $\widehat{\Sigma}_{FMoE}$. This yields the following result:

\begin{thm}[Exponential concentration of median of sample covariance matrices]\label{thm_exp_conc_cov_mat}
    Under the same setting in Proposition~\ref{prop_tail_of_cov_matrix}, we set $\widehat{\Sigma}_{FMoE}$ as in Algorithm~\ref{alg_moe} with the original estimator being a sample covariance matrix and the metric $d$ being either $d_{AI}$ or $d_{BW}$. Set $k = \floor{\log(1/\delta)/\psi(0.4, 0.1)} + 1$.
    \begin{compactitem}
    \item[(a)] For $d = d_{AI}$,  whenever $n \geq 2 k C d^4 / \lambda_{\min}^2$, we have
    \begin{align*}
    \bbP &\bigg[ d_{AI}(\widehat{\Sigma}_{FMoE}, \Sigma) \geq \\
    &- 1.3 \sqrt{d} \log\left(1 - \frac{9 d^2}{\lambda_{\min}} \sqrt{\frac{C \log(1.4/\delta)}{n}}\right) \bigg] \leq \delta.
    \end{align*}
    \item[(b)] For $d = d_{BW}$, again assume both $\widehat{\Sigma}$ and $\Sigma$ have the eigenvalue lower bound by $\lambda_{0} > 0$. In addition assume conditions in Theorem \ref{thm_concen_of_med}(b) holds with $\kappa = 3/(2\lambda_{0}^2)$, $\theta = \Sigma$, and $\widehat{\theta}_j = \widehat{\Sigma}_j$. Then, whenever $n > 6 \lambda_0 k C d^4$, we have
    \begin{align*}
        \bbP &\left[d_{BW}(\widehat{\Sigma}_{FMoE}, \Sigma) > 12d^2 \sqrt{\frac{C \log(1.4/\delta)}{2 n \lambda_{0}}} \right] \leq \delta.
    \end{align*}
    \end{compactitem}
\end{thm}
We provide a proof in Appendix \ref{appendix_proof_section_applications}.

\begin{rmk}
\begin{enumerate}
    \item []
    \item Note that $d$ can vary in Proposition~\ref{prop_tail_of_cov_matrix}; as long as $d^4 = o(n/\log n)$ Proposition~\ref{prop_tail_of_cov_matrix} ensures polynomial concentration. Therefore, our method still achieves the exponential concentration in such high dimensional settings. 
    \item Our estimator is comparable to estimators studied in \citet{abdalla2023covarianceestimationoptimaldimensionfree, oliveira2024improvedcovarianceestimationoptimal}, while their estimators require some additional (still weak) assumptions. For example, $d_{BW}$ bound of the estimator in \citet{oliveira2024improvedcovarianceestimationoptimal} $\widetilde{\Sigma}$ (considering 0 contamination) will be
    \[
    \mathbb{P}\left(d_{BW}(\widetilde{\Sigma}, \Sigma) \geq \frac{C \lambda_{\max} \sqrt{d}}{2} \sqrt{\frac{r(\Sigma) + \log(1/\delta)}{n \lambda_0}}\right) \leq \delta.
    \]
    Here, $r(\Sigma)$ denotes the effective rank of $\Sigma$. This bound is more desirable with respect to dimension $d$, but has the additional factor $\lambda_{\max}$, and the additive term $\frac{C \lambda_{\max} \sqrt{\lambda_{\max} r(\Sigma)}}{2\sqrt{n \lambda_0}}$. 
    Furthermore, note the dimension dependencies of our bounds come from Proposition \ref{prop_tail_of_cov_matrix}. If one uses a different original estimator, or obtains the tighter dimension bounds of $\widehat{\Sigma}$, $\widehat{\Sigma}_{FMoE}$ can exhibit a tighter concentration in terms of the dimension. A similar analysis can be conducted for $d_{AI}$ metric bound.
\end{enumerate}
\end{rmk}

The result of Theorem \ref{thm_exp_conc_cov_mat} shows that we can obtain the estimator with respect to $d_{AI}$ and $d_{BW}$ that achieves the exponential concentration only with the moment assumptions. On the other hand, the proof of Proposition \ref{prop_tail_of_cov_matrix} indicates that the original sample covariance matrix can achieve the exponential concentration with respect to $d_{AI}$ and $d_{BW}$ when it exhibits the exponential concentration with respect to matrix norms, which is possible typically under the sub-Gaussian type assumption \citep{vershynin2012introduction}[Corollary 5.50].

\section{Implementation and experiments}\label{section_implementation_experi}

A general procedure to implement our method is displayed in Algorithm \ref{alg_moe}. 
To implement the algorithm in practice, the most important part is the choice of the number of blocks $k$. While there is a precise quantity of $k$ for the given the confidence rate $\delta$ as derived in Section \ref{section_applications}, in practice $k$ is chosen to guarantee $p$ in Theorem \ref{thm_concen_of_med} to be small enough for $\floor{n/k}$ number of data \citep{lin2020RobustOA}[Remark 3.1]. Following this philosophy, the choice of $k$ for each experiment is determined to ensure the original estimator to reasonably concentrate within $\floor{n/k}$ data. 

We conducted experiments for each example in Section~\ref{section_applications}: Fr\'echet mean estimation in a NPC space and boosting the sample covariance estimator with respect to two different metrics $d_{AI}$ and $d_{BW}$. 

\begin{figure}[!t]
\centering
\includegraphics[width=0.18\textwidth]{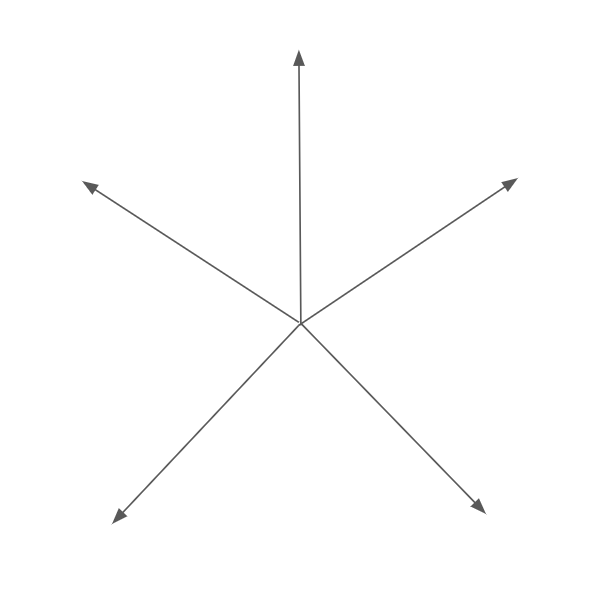}\hfill
\includegraphics[width=0.30\textwidth]{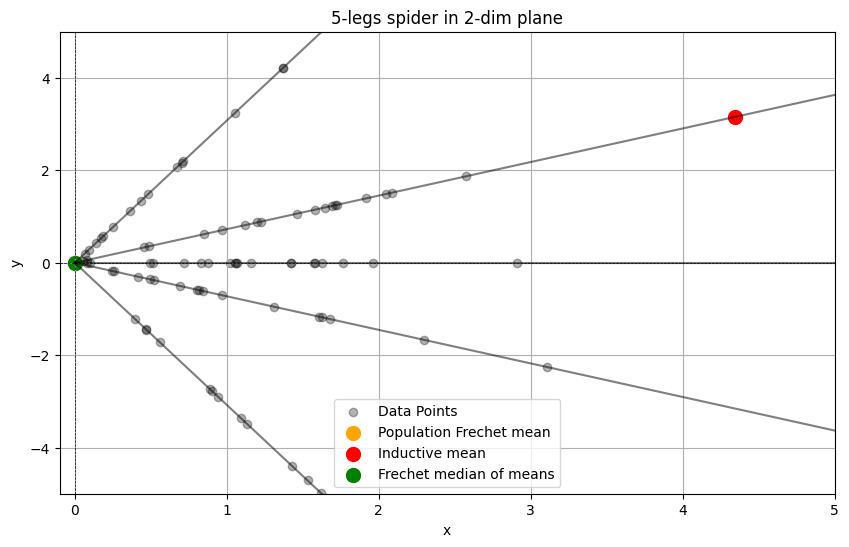}
%
\caption{\textbf{Left}: A 5-legs spider tree. \textbf{Right}: One experiment result on the 5-legs spider tree. The origin denotes the population Fr\'chet mean, and red and green dot stand for inductive mean and Fr\'echet median of inductive means respectively.}\label{figure_spider_5}
\end{figure}

\begin{figure*}[!h]
\centering
\includegraphics[width=0.33\textwidth]{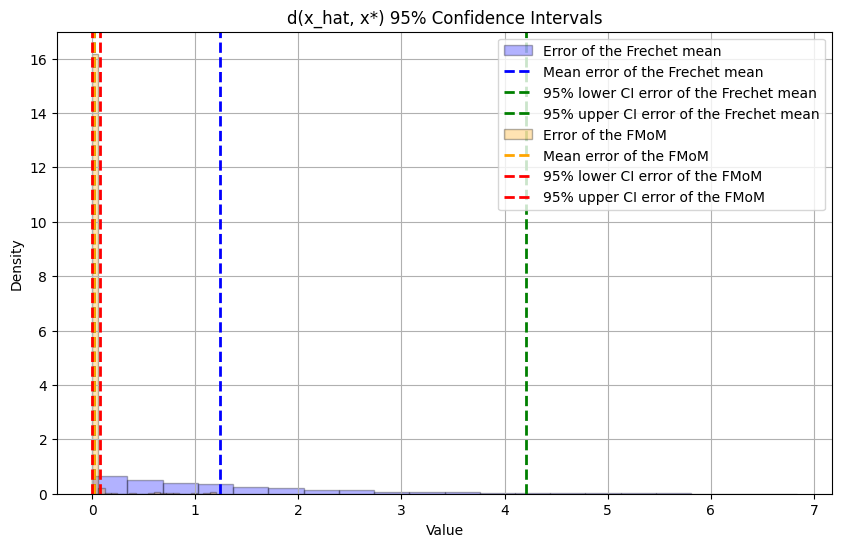}\hfill
\includegraphics[width=0.33\textwidth]{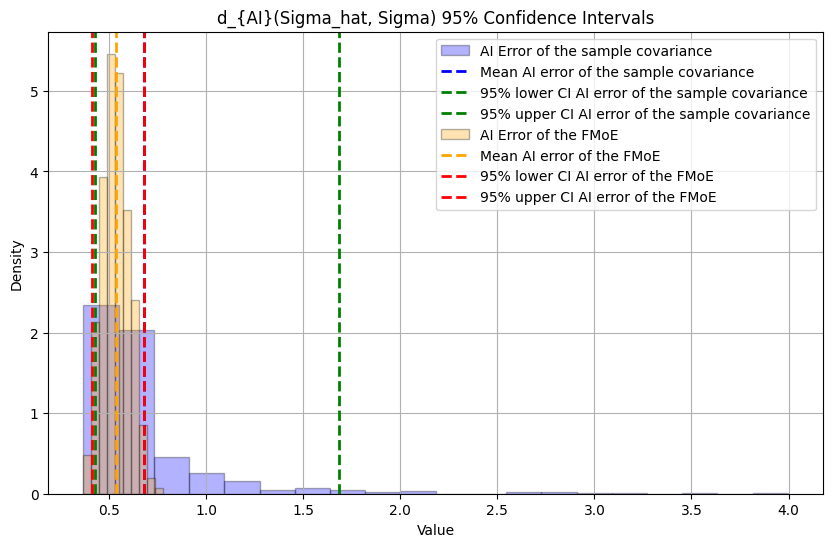}\hfill
\includegraphics[width=0.33\textwidth]{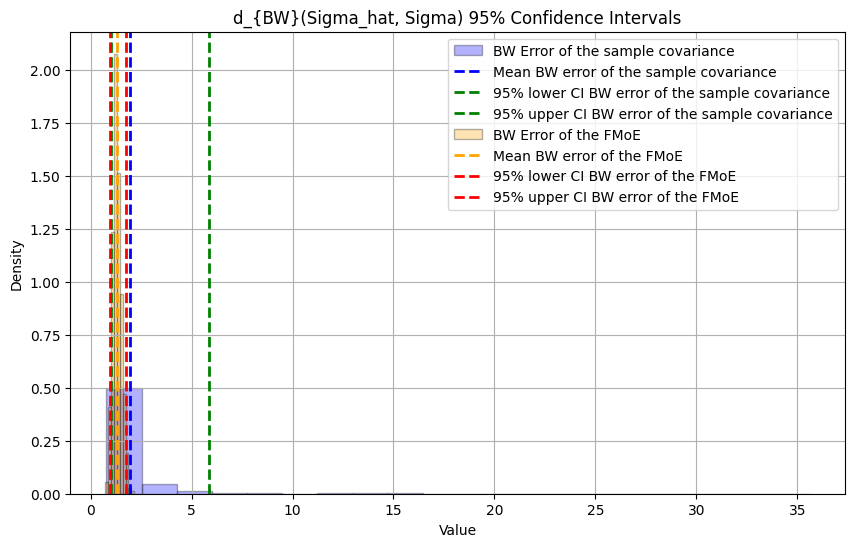}
%
\caption{Histogram, mean, and 95\% confidence interval for each experiment from 1000 simulations. \textbf{Left}: 5-legs spider. \textbf{Middle}: $(SPD, d_{AI})$. \textbf{Right}: $(SPD, d_{BW})$. All results indicate our method achieves much stronger concentration as well as much smaller mean squared errors.}\label{figure_res_ci}
\end{figure*}


\begin{table*}[!h]
\caption{Mean squared error and 95\% confidence interval comparisons from 1000 simulations.} \label{table_experi_res}
\begin{center}
\begin{tabular}{c|cc|cc}
Task  & $\bbE d^2(\widehat{\theta}, \theta)$ & $\bbE d^2(\widehat{\theta}_{FMoE}, \theta)$ & $d(\widehat{\theta}, \theta)$ CI & $d(\widehat{\theta}_{FMoE}, \theta)$ CI\\
\hline
5-spider tree & 3.1244 & $1.2 \times 10^{-5}$ & $[0.0110, 4.4202]$ & $[9.6 \times 10^{-5}, 0.0086]$\\
Covariance ($d_{AI}$) & 0.6057 & 0.2931 & $[0.4266, 1.6865]$ & $[0.4132, 0.6792]$\\
Covariance ($d_{BW}$) & 8.3281 & 1.7360 & $[0.9936, 5.8796]$ & $[0.9443, 1.7409]$
\end{tabular}
\end{center}
\end{table*}




As a first experiment, we conducted Fr\'echet mean estimation problem in a metric tree, a NPC space that is neither a Riemannian manifold nor a Hilbert space. Metric trees are widely used for their theoretical versatility and practical applications \citep{fakcharoenphol2003atightbound, abuata2016metrictree}. In this regard, estimating Fr\'echet mean in metric trees is a frequently encountered problem  \citep{romon2023convex}. Among the various choices of metric trees, we chose a spider tree model. A spider tree is the simplest metric tree that can be embedded in $\bbR^2$. A $d$-legs spider tree $S_d$ consists of $d$-copies of the positive real line, called legs, glued together at the origin. Consequently, this space is also an example of a stratified space. Figure \ref{figure_spider_5} illustrates a 5-legs spider tree. Mathematically, $S_d$ can be viewed as a quotient space $\set{1,\dots,d} \times \bbR_{\geq 0} / \sim $, where the equivalence relation is defined by $(x_1, 0) \sim (y_1, 0)$ for all $x_1, y_1 \in \set{1, \dots, d}$. The metric in this space is defined as follows: for $(x_1, x_2), (y_1, y_2) \in S_d$
\[
d(x,y) = \left\{
\begin{array}{ll}
    \abs{x_2-y_2} & \text{if } x_1 = y_1,\\
    \abs{x_2} + \abs{y_2} & \text{otherwise}.\\
\end{array} 
\right.
\]
The equivalence class $(x_1, 0)$ is called the center node. For more explanation on metric trees and spider trees, we refer to \citet{aksoy2010someresult}.

We chose a $5$-legs spider tree for our experiment. For the probability measure, we used $P = \text{Unif}(\set{1, \dots, 5}) \times \left((1-\alpha) \abs{N(1,1)} + \alpha \abs{N(100,1)}\right)$. Here, $\abs{N(\cdot, \cdot)}$ is a distribution of $\abs{X}$ for $X \sim N(\cdot, \cdot)$. $\alpha$ was used to add a small portion of outliers to make the distribution heavy tail. Due to symmetry, the Fr\'echet mean of $P$ becomes the center node. We used $\alpha = 0.1$, a sample size $n = 100$, and the number of blocks $k = 10$ for this experiment. Figure \ref{figure_spider_5} shows one of the simulation results.

For the covariance estimation problem, we set the population distribution in dimension $d = 10$ as $t_{2.5}(0, \Sigma)$,
where $\Sigma$ was randomly generated while fixing eigenvalues $\lambda_j = j$ for $j= 1, \dots d$. This distribution has a polynomial tail with a covariance being $5\Sigma$. We used a sample size of $n = 10 d^4$ (motivated by the $d^4$ term in the bound of Proposition \ref{prop_tail_of_cov_matrix}) and the number of blocks $k = 5$ for both $d_{AI}$ and $d_{BW}$.

Our results from representative experiments are summarized in Table \ref{table_experi_res} and Figure \ref{figure_res_ci}. For all experiments in this section, we conducted 1000 simulations for each task and obtained average squared errors and 95\% confidence intervals. Overall, the experiment results indicate that our method achieves higher concentration when the distribution has a heavy tail, provided the block size is chosen appropriately. Additional experiments under different settings (block sizes, population distributions, model spaces) are provided in Appendix~\ref{appendix_additional_experiment}.

Codes for our experiments can be found at \url{https://github.com/wldyddl5510/Frechet_median_of_means/}, and the implementation details are provided in Appendix \ref{appendix_implementation_detail}.

\section{Conclusion}

In this work, we extend the Fr\'echet median of estimators to CAT($\kappa$) spaces, which include almost all spaces of interest in statistics and machine learning. This generalization allows us to obtain exponential concentration on CAT($\kappa$) spaces under mild assumptions. In particular, we leverage this result to famous Fr\'echet mean and covariance estimation problems. Lastly, supportive numerical evidences are also provided. We conclude the paper with some open questions.

\begin{compactitem}
    \item[1.] Is there a more general space that enables Fr\'echet median based robust estimation to work? Our analysis on CAT($\kappa$) spaces almost covers the all existing approaches of Fr\'echet median based robust estimation methods, but not Banach space case proposed in \citet{minsker2015geometric}. This implies there is a possibility of more generalizations, for example, to non-Riemannian Finsler manifolds which are not Alexandrov spaces. In NPC spaces, the concept of generalized CAT(0) \citep{khamsi2017generalizedcat0} incorporates the CAT(0) spaces and Banach spaces. However, we are not aware of the generalization of this concept beyond NPC spaces.
    \item[2.] Is there a way to overcome the curvature upper bound condition? Our proof crucially relies on the curvature upper bound. However, some spaces without the curvature upper bound are widely studied, e.g., Wasserstein space over $\mathbb{R}^d$. Whether one can extend the similar methods to such spaces is unknown.
    \item[3.] Can one obtain bounds independent of population quantities? When one wants to determine the optimal number of blocks or construct the confidence regions, unknown population quantities in the bound hinder us from obtaining the precise value. Some methods were developed to overcome those dependencies, e.g., adaptation methods. Developing such a method for our estimator is an important factor for practical applications.
\end{compactitem}




\subsubsection*{Acknowledgements}

We appreciate Young-Heon Kim, Inyoung Ryu, and Byeongsu Yu for their helpful discussions. In addition, we appreciate anonymous reviewers for their insightful feedback during the review process. JP acknowledges support from NSF DMS 2210689. AB acknowledges NSF DMS 2210689, NSF DMS 1916371 for supporting this project.


\newpage

\clearpage

\newpage

\subsubsection*{References}

\bibliographystyle{plainnat}

\bibliography{main.bib}


\newpage

\clearpage

\newpage

\section*{Checklist}

 \begin{enumerate}

 \item For all models and algorithms presented, check if you include:
 \begin{enumerate}
   \item A clear description of the mathematical setting, assumptions, algorithm, and/or model. [\checkmark Yes /No/Not Applicable]

    - We included mathematical detail in Sections~\ref{section_robustify} and \ref{section_applications}.
   
   \item An analysis of the properties and complexity (time, space, sample size) of any algorithm. [\checkmark Yes/No/Not Applicable]

    - We included the a brief description of the algorithm when the algorithm is feasible in the end of Section \ref{section_robustify}. However, we did not include the complexity, as it depends on the choice of specific implementation detail, e.g., which Fr\'echet median algorithm to use or what the parameter space is.
   
   \item (Optional) Anonymized source code, with specification of all dependencies, including external libraries. [\checkmark Yes/No/Not Applicable]

   - We will attach the source code in the supplement.
 \end{enumerate}

 \item For any theoretical claim, check if you include:
 \begin{enumerate}
   \item Statements of the full set of assumptions of all theoretical results. [\checkmark Yes/No/Not Applicable]

    - Sections \ref{section_robustify} and \ref{section_applications} specifically addressed theoretical detail with full explanations.
   
   \item Complete proofs of all theoretical results. [\checkmark Yes/No/Not Applicable]

    - All proofs are provided in Appendix \ref{appendix_proofs}.
   
   \item Clear explanations of any assumptions. [\checkmark Yes/No/Not Applicable]    

    - We made remarks for the most of Theorems discussing the assumptions.
   
 \end{enumerate}

 \item For all figures and tables that present empirical results, check if you include:
 \begin{enumerate}
   \item The code, data, and instructions needed to reproduce the main experimental results (either in the supplemental material or as a URL). [\checkmark Yes/No/Not Applicable]

    - We will attach the source code in the supplement, which contains the exact code one can reproduce.
   
    \item All the training details (e.g., data splits, hyperparameters, how they were chosen). [\checkmark Yes/No/Not Applicable]

    - We provided implementation and experiment detail in Appendix \ref{appendix_additional_experiment}.
    
    \item A clear definition of the specific measure or statistics and error bars (e.g., with respect to the random seed after running experiments multiple times). [\checkmark Yes/No/Not Applicable]

    - We explicitly mentioned the measuring criteria and number of simulations in Section \ref{section_implementation_experi} and Appendix \ref{appendix_additional_experiment}.
    
    \item A description of the computing infrastructure used. (e.g., type of GPUs, internal cluster, or cloud provider). [\checkmark Yes/No/Not Applicable]

    - We explicitly mentioned the computing resource we used in Appendix \ref{appendix_additional_experiment}.
    
 \end{enumerate}

 \item If you are using existing assets (e.g., code, data, models) or curating/releasing new assets, check if you include:
 \begin{enumerate}
   \item Citations of the creator If your work uses existing assets. [\checkmark Yes/No/Not Applicable]

    - We made a citation of the package we used in Appendix \ref{appendix_additional_experiment}.
   
   \item The license information of the assets, if applicable. [Yes/No/\checkmark Not Applicable]

    - We only utilized the open-source libraries.
   
   \item New assets either in the supplemental material or as a URL, if applicable. [Yes/No/\checkmark Not Applicable]
   \item Information about consent from data providers/curators. [Yes/No/\checkmark Not Applicable]
   
   \item Discussion of sensible content if applicable, e.g., personally identifiable information or offensive content. [Yes/No/\checkmark Not Applicable]
 \end{enumerate}

 \item If you used crowdsourcing or conducted research with human subjects, check if you include:
 \begin{enumerate}
   \item The full text of instructions given to participants and screenshots. [Yes/No/\checkmark Not Applicable]
   \item Descriptions of potential participant risks, with links to Institutional Review Board (IRB) approvals if applicable. [Yes/No/ \checkmark Not Applicable]
   \item The estimated hourly wage paid to participants and the total amount spent on participant compensation. [Yes/No/\checkmark Not Applicable]
 \end{enumerate}

 \end{enumerate}

\onecolumn
\aistatstitle{Robust Estimation in metric spaces: Achieving Exponential Concentration with a Fr\'echet Median: \\
Supplementary Materials}
\appendix
\section{Preliminaries}
\subsection{CAT($\kappa$) spaces} \label{appendix_alexandrov_catk}
CAT($\kappa$) spaces, also known as Alexandrov spaces, are a generalization of Riemannian manifolds with a bounded upper sectional curvature. CAT($\kappa$) spaces are locally compact complete length spaces with a uniformly bounded curvature. Due to its generality and rich implications in geometry, they have been extensively studied more than seven decades (\citet{Aleksandrov1951, BuragoGromovPerelman1992, BridsonHaefliger1999, sturm1999metric, Buragoetalmetrygeometry2001}).

There are three equivalent ways to define them via the \emph{distance functions}, the \emph{speed of geodesics} and the \emph{comparison triangles}, respectively. For the sake of completeness, we introduce all of them. For more details, see \citet{kunzinger2018alexndrov}.

To begin with, we define model spaces $M^2_\kappa$.

\begin{defn}[Model spaces]
For $\kappa \in \mathbb{R}$, we denote by $M^2_\kappa$ the following spaces:
\begin{enumerate}
    \item[(i)] $M^2_0$ is the Euclidean plane $\mathbb{R}^2$.
    \item[(ii)] For $\kappa > 0$, $M^2_\kappa = \frac{1}{\sqrt{\kappa}} \mathbb{S}^2$ where $\mathbb{S}^2$ is the $2$-dimensional sphere.
    \item[(iii)] For $\kappa < 0$, $M^2_\kappa = \frac{1}{\sqrt{|\kappa|}} \mathbb{H}^2$ where $\mathbb{H}^2$ is the $2$-dimensional hyperbolic space.
\end{enumerate}
We use $d_\kappa$ to denote the intrinsic metric of $M^2_\kappa$.
\end{defn}

Throughout this section, let $(\calX, d)$ be a complete, connected metric space.


\begin{defn}[Length space]
$(\calX, d)$ is called a length space if
the metric $d$ is the intrinsic length metric, i.e. for any $x, y \in \calX$, and any path $\alpha : [a,b] \to \calX$ such that $\alpha(a)=x, \alpha(b)=y$,
\[
    d(x,y) = \inf_{\alpha} \left\{ \emph{length } \alpha \right\}
\]
where the length of path $\alpha$ is defined as
\[
    \emph{length } \alpha := \sup \left\{ \sum_{i \geq 1} d(\alpha(t_i), \alpha(t_{i-1})) : a=t_0 \leq t_1 \leq \dots \leq t_n=b \right\}.
\]
\end{defn}

\begin{defn}[Complete length space]
A path $\alpha : [a,b] \to \calX$ is called a shortest path if for any other paths $\beta$ connecting $\alpha(a)$ to $\alpha(b)$, $\emph{length } \alpha \leq \emph{length } \beta$.

For a length space $(\calX, d)$, $\alpha$ is a shortest path connecting $\alpha(a)$ to $\alpha(b)$ if and only if
\[
    d( \alpha(a),\alpha(b)) = \emph{length } \alpha.
\]
If for any $x,y \in \calX$, there is a shortest path connecting $x$ to $y$, then $(\calX, d)$ is called strictly intrinsic, and a complete length space. 
\end{defn}

\begin{defn}[Geodesic space]\label{def: geodesic space}
We call a complete length space as a geodesic space, and shortest paths as geodesics.    
\end{defn}

\begin{defn}
For $x,y \in \calX$, write $T = d(x,y)$. $\alpha$ is called a geodesic with unit speed connecting $x$ to $y$, or a geodesic parametrized by arc length, if $\alpha : [0, T] \to \calX$ satisfies $d(x, \alpha(t) ) = t$ for any $t \in [0,T]$.

We use $[xy]=\alpha$ to denote a geodesic with unit speed connecting $x$ to $y$.
\end{defn}

Now, we provide the first definition of CAT($\kappa$) spaces via distance function. In simple terms, this definition states that curvature distorts the space such that, from any fixed point, any point lies in the middle of the geodesic is further away than it would be in the model space.

Given $x,y \in \calX$, fix $p \in \calX$. \emph{One-dimensional distance function} (at $p$) is defined as
\[
    g(t):= d(p, \alpha(t)).
\]

We will compare $g(t)$ to an appropriate one-dimensional distance function in the model space. To this end, given $x,y, p \in \calX$, choose $\widetilde{x}, \widetilde{y}, \widetilde{p}$ in the model space $M_{\kappa}^2$ such that
\[
    d_\kappa(\widetilde{x}, \widetilde{y}) = d(x,y),\, d_\kappa(\widetilde{p}, \widetilde{x}) = d(p,x) \text{ and } d_\kappa(\widetilde{p}, \widetilde{y}) = d(p,y).
\]
For $\kappa > 0$ we further assume 
\begin{equation}\label{eq: perimeter condition for positive kappa}
    d(x,y) + d(p,x) + d(p,y) \leq 2 D_\kappa = \frac{2 \pi}{\sqrt{\kappa}}.
\end{equation}
We call $\{ \widetilde{x}, \widetilde{y}, \widetilde{p} \}$ as \emph{comparison configuration}, $\widetilde{p}$ as a \emph{reference point}, and $[\widetilde{x} \widetilde{y}]=\widetilde{\alpha} : [0,T] \to M^2_\kappa$, the geodesic connecting $\widetilde{x}$ to $\widetilde{y}$ with unit speed in $M^2_\kappa$, as a \emph{comparison segment} respectively. Note that comparison configuration is unique up to rigid motions.

\begin{defn}[Comparison distance function]
Define
\[
    \widetilde{g}_{\kappa}(t):= d_\kappa(\widetilde{p}, \widetilde{\alpha}(t))
\]
the model space distance from the reference point $\widetilde{p}$ to the comparison segment $[\widetilde{x} \widetilde{y}]=\widetilde{\alpha}$. This $\widetilde{g}_{\kappa}$ is called the comparison distance function of $g$.
\end{defn}

\begin{defn}[Distance condition]\label{def: distance condition}
A geodesic space $(\calX, d)$ is a CAT($\kappa$) space if every point in $\calX$ has a neighbourhood $U$ such that the following holds: for any point $p \in U$ and all $[xy]=\alpha \subseteq U$, the comparison distance function $\widetilde{g}_{\kappa}(t)$ for one-dimensional distance function $g(t)=d(p, \alpha(t))$ satisfies
\begin{equation}\label{eq: distance condition}
     \widetilde{g}_{\kappa}(t) \geq g(t) \text{ for all $t \in [0,T]$}.
\end{equation}
For $\kappa > 0$, we further assume \eqref{eq: perimeter condition for positive kappa}.
\end{defn}

The next definition is followed by triangle comparison. In fact, it is straightforwardly equivalent to the definition via distance function.

For $x,y,z \in \calX$, a \emph{triangle} is a triangle with sides $[xy], [yz], [xz]$, each of which is a geodesic with unit speed of a pair of three points. We write $\triangle xyz$ to denote for the triangle of $x,y,z$ with side lengths $d(x,y), d(y,z)$ and $d(x,z)$. Notice that since there are multiple geodesics, $x,y,z$ cannot determine $\triangle xyz$ uniquely. However, every $\triangle xyz$ has the same side lengths.

A \emph{comparison triangle} $\triangle \widetilde{x}\widetilde{y}\widetilde{z}$ is a triangle of a comparison configuration in the model space; hence $\triangle \widetilde{x}\widetilde{y}\widetilde{z}$ has the same side lengths as $\triangle xyz$. Clearly, a comparison triangle is also unique up to rigid motions. Let $[\widetilde{x} \widetilde{y}]$ be the geodesic in the model space connecting $\widetilde{x}$ to $\widetilde{y}$ with the length $d(x,y)$: i.e. the side between $\widetilde{x}$ and $\widetilde{y}$ of the comparison triangle $\triangle \widetilde{x}\widetilde{y}\widetilde{z}$. If $\alpha=[xy]$, we use $\widetilde{\alpha}:= [\widetilde{x} \widetilde{y}]$, the natural comparison geodesic in the model space induced by $\alpha$.

\begin{defn}[Triangle condition]\label{def: triangle condition}
A geodesic space $(\calX, d)$ is CAT($\kappa$) space if every point in $\calX$ has a neighbourhood $U$ such that the following holds: for every triangle $\triangle xyz \subseteq U$ and every point $w \in [xz]$,
\begin{equation}\label{eq: triangle condition}
    d(w, y) \leq d_\kappa(\widetilde{w}, \widetilde{y})
\end{equation}
where $\widetilde{w} \in [\widetilde{x} \widetilde{z}]$ such that $d_\kappa(\widetilde{x}, \widetilde{w}) = d(x,w)$, or $\widetilde{w}=\widetilde{\alpha}(d(x,w))$. For $\kappa > 0$, we further assume \eqref{eq: perimeter condition for positive kappa}.
\end{defn}

The last one to define CAT($\kappa$) spaces is to compare angles. In words, angles in upper bounded curvatured spaces are smaller than in the model space. This is indeed equivalent to the fact that two geodesics in CAT($\kappa$) spaces starting at the same initial point with a certain angle is indeed further than those with the same angle in the model space.

\begin{defn}
Given distinct points $x,y,z \in \calX$, the comparison angle $\angle \widetilde{x} \widetilde{y} \widetilde{z}$ at $y$ is the angle at $\widetilde{y}$ of the comparison triangle $\triangle \widetilde{x} \widetilde{y} \widetilde{z}$. Alternatively, 
\[
    \angle \widetilde{x} \widetilde{y} \widetilde{z} := \arccos{\frac{d(x,y)^2 + d(y,z)^2 - d(x,z)^2}{2d(x,y)d(y,z)}}.
\]
\end{defn}

\begin{defn}
Let $\alpha, \beta : [0,T) \to \calX$ be two paths with the same initial point $p$. The angle between $\alpha$ and $\beta$ is defined as
\[
    \angle(\alpha, \beta) := \lim_{s,t \to 0} \angle \widetilde{\alpha}(s) \widetilde{p} \widetilde{\beta}(t)
\]
if the limit exists. Given three distinct points $x,y,z \in \calX$, the angle of $\triangle xyz$ at $y$ is defined as
\[
    \angle xyz := \angle([xy], [yz]).
\]
\end{defn}

\begin{rmk}
It is not trivial that the angle between two paths always exists. However, if $\kappa \leq 0$, it always exists: see \citet{kunzinger2018alexndrov}[3.3.2 Proposition]. If $\kappa > 0$, \eqref{eq: perimeter condition for positive kappa} should be required.
\end{rmk}

\begin{defn}[Angle condition]\label{def: angle condition}
A geodesic space $(\calX, d)$ is CAT($\kappa$) space if every point in $\calX$ has a neighbourhood $U$ such that the following holds: for every triangle $\triangle xyz \subseteq U$ the angles of $\triangle xyz$ satisfy
\begin{equation}\label{eq: angle condition}
    \angle yxz \leq \angle \widetilde{y} \widetilde{x} \widetilde{z}, \, \angle zyx \leq \angle \widetilde{z} \widetilde{y} \widetilde{x}, \text{ and } \angle xzy \leq \angle \widetilde{x} \widetilde{z} \widetilde{y}
\end{equation}
where $\triangle \widetilde{x}  \widetilde{y} \widetilde{z}$ is a comparison triangle of $\triangle xyz$ in the model space $M^2_\kappa$. For $\kappa > 0$, we further assume \eqref{eq: perimeter condition for positive kappa}.
\end{defn}

The three conditions, \eqref{eq: distance condition}, \eqref{eq: triangle condition} and \eqref{eq: angle condition}, are equivalent to define CAT($\kappa$) spaces.

\begin{thm}
All the definitions of CAT($\kappa$) spaces, that is the distance condition~\ref{def: distance condition}, the triangle condition~\ref{def: triangle condition}, and the angle condition~\ref{def: angle condition} are equivalent.
\end{thm}


Lastly, we note that analogous statements can be made for spaces with bounded lower curvature by reversing the direction of the inequalities in the definitions above. Such spaces arise in the fields of optimal transport and Wasserstein geometry.

\subsubsection{Fr\'echet mean estimation in NPC spaces}\label{appendix_frechet_mean_estimation}

In this appendix, we consider two different estimators for population Fr\'echet mean in NPC spaces: empirical Fr\'echet mean and inductive mean. Throughout this appendix, we fix $(\calX, d)$ to be an NPC space.

First is an empirical Fr\'ehet mean, the most natural way to estimate the Fr\'echet mean. 
\begin{defn}[Empirical Fr\'echet mean]
Given $n$ data points $x_1, \dots, x_n \in \calX$, the empirical Fr\'echet mean is defined by 
\[
\widehat{x}_{EM} = \argmin_{x \in \calX} \frac{1}{n} \sum_{j=1}^{n} d^2(x, x_j).
\]
\end{defn}

Inductive mean is another natural way to estimate the Fr\'echet mean proposed by \citet{sturm2000NPC}, coming from the generalization of the law of large numbers. As the name implies, the inductive mean is defined `inductively'. 
\begin{defn}[Inductive mean]
    Set $\delta$ and $k$ in a same way. Given $n$ data points $x_1, \dots, x_n \in \calX$, define a sequence $s_i$ as follows:
    \[
    s_1 = x_1, \qquad s_i = \left(1 - \frac{1}{i}\right)s_{i-1} + \frac{1}{i} x_i \text{ for } i = 2, \dots, n
    \]
    where the summation can be understood as a geodesic interpolation with a given ratio. Then, the resulting $s_n := \widehat{x}_{IM}$ is called inductive mean.
\end{defn}

These two estimators coincide to the arithmetic mean in Hilbert space, but not in the general metric space. As pointed out in \citet{sturm2000NPC}, the inductive mean depends on the permutation of $x_i$'s, unlike the empirical Fr\'echet mean. However, the inductive mean has certain advantages over the empirical Fr\'echet mean from both theoretical and practical perspectives. Practically, if you have a geodesic interpolation oracle, computing the inductive mean is straightforward. Theoretically, the following proposition highlights the benefits of using the inductive mean.

\begin{prop} 
    Let $x_i, \dots, x_n \iid P \in \calP_2(\calX)$. Let $\widehat{x}$ be either $\widehat{x}_{EM}$ or $\widehat{x}_{IM}$. If $\widehat{x} = \widehat{x}_{EM}$, in addition assume $\kappa_{\min}(\calX) > -\infty$. Then,
    \[
    \bbE\left[d^2(\widehat{x}, x^*)\right] \leq \frac{\sigma^2}{n}
    \]
    where $\sigma^2 = \bbE_{y \sim P}[d^2(x^*, y)]$ is the second moment of $P$.
\end{prop}
\begin{proof}
    For $\widehat{x}_{EM}$ case, see \citet{LeGouic2019FastCO}[Corollary 3.4]. For $\widehat{x}_{IM}$ case, see \citet{sturm2000NPC}[Theorem 4.7].
\end{proof}

The fact that $\widehat{x}_{IM}$ does not require the additional condition on the curvature lower bound of $\calX$ is favorable, as there are cases where we need to deal with spaces of unbounded curvature, such as metric trees or statistical manifolds with Gamma and Dirichlet distributions, to name a few.


\section{Deferred proofs} \label{appendix_proofs}

This appendix contains detailed proofs of the results that are missing in the main paper.

\subsection{Proofs in Section \ref{section_robustify}} \label{appendix_proof_section_robustify}

\textbf{Proof of Lemma \ref{lem_geo_med}:}
\begin{proof}

As in \citet{minsker2015geometric}[Lemma 2.1] suppose the implication is false for the contradiction. Without the loss of generality, it means that $d(x_j, z) \leq r$ for  $j = 1, \dots, \left \lfloor{(1-\alpha) k}\right \rfloor + 1$.

We separately analyze the cases when $\kappa \leq 0$ and $\kappa > 0$.

\textbf{CASE I: $\kappa \leq 0$.}

Recall
\[
    F(x) := \frac{1}{k} \sum_{j=1}^k  d(x, x_j).
\]
Note that $F(x)$ always admits a minimizer followed by \citet{bacak2014computing}[Lemma 2.3]. Let $x^*$ be a median (a minimizer of $F(\cdot)$) and $z \neq x^*$ be an arbitrary point. Let $\alpha: [0,1] \rightarrow \calX$ be a geodesic curve in $\calX$ such that $\alpha(0) = x^*$ and $\alpha(1) = z$. Since $\alpha(0) = x^*$ is a minimizer of $F$, we have
    \begin{align}\label{eq_minimizer_condi}
        \limsup_{t \rightarrow 0} \frac{F(\alpha(t)) - F(\alpha(0))}{t} \geq 0.
    \end{align}
Now, notice that 
    \begin{align}\label{eq_first_var}
        \begin{split}
        \limsup_{t \rightarrow 0} \frac{F(\alpha(t)) - F(\alpha(0))}{t} &\leq \sum_{j=1}^{k}  \limsup_{t \rightarrow 0}\frac{d(x_j, \alpha(t)) - d(x_j,x^*)}{t}  \indicator_{\{x_j \neq x^*\}}\\
        &\quad + \sum_{j=1}^{k}  \limsup_{t \rightarrow 0} \frac{d(x_j, \alpha(t))}{t} \indicator_{\{x_j = x^*\}}.
        \end{split}
    \end{align}

    For the first term, the first variation formula in Alexandrov spaces (see \cite{kunzinger2018alexndrov}[Proposition 3.4.2]) with $l(t) = d(x_j, \alpha(t))$ gives
    \[
    \limsup_{t \rightarrow 0}\frac{d(x_j, \alpha(t)) - d(x_j,x^*)}{t} \leq - \cos \gamma_j
    \]
    where $\gamma_j = \angle x_j x^* z$ in Alexandrov sense. Since $\calX$ is a NPC space, comparing between $\triangle x_j x^* z$ and its Euclidean comparison triangle $\triangle 
    \widetilde{x}_j \widetilde{x}^* \widetilde{z}$ (see Definitions \ref{def: triangle condition} and \ref{def: angle condition}) results in
    \begin{align}\label{eq_eucli_compare}
        \gamma_j \leq \widetilde{\gamma}_j := \angle \widetilde{x}_j \widetilde{x}^* \widetilde{z}, \quad d(x_j, \alpha(t)) \leq \|\widetilde{x}_j - \widetilde{\alpha}(t)\|.
    \end{align}
    where $\widetilde{\alpha} = [\widetilde{x}^* \widetilde{z}]$. Now, since $\widetilde{\gamma}_j$ is the angle inside the triangle, $\widetilde{\gamma}_j < \pi$ holds. This implies $\cos(\gamma_j) \geq \cos(\widetilde{\gamma}_j)$.
    Plugging-in Equation \eqref{eq_eucli_compare} to Equation \eqref{eq_first_var} yields
    \begin{align*}
        \limsup_{t \rightarrow 0} \frac{F(\alpha(t)) - F(\alpha(0))}{t} \leq -\sum_{j=1}^{k}  \cos(\widetilde{\gamma}_j)\mathds{1}_{\{\widetilde{x}_{j} \neq \widetilde{x}^*\}} + \sum_{j=1}^{k}  \mathds{1}_{\{\widetilde{x}_{j} = \widetilde{x}^*\}} < -(1-\alpha) k \sqrt{1 - \frac{1}{C_{\alpha}^2}} + \alpha k \leq 0
    \end{align*}
    where the second inequality follows in the same manner as \citet{minsker2015geometric}[Lemma 2.1]; see Figure \ref{figure_npc_triangle} as well. Since it contradicts to Equation \eqref{eq_minimizer_condi}, we prove the claim.

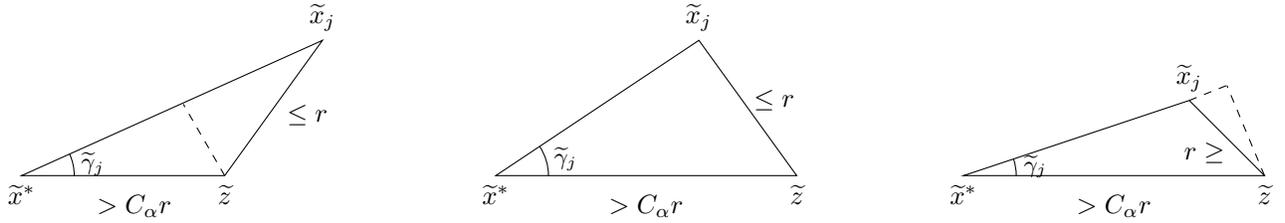
\begin{figure}[h] 
    \centering
    \begin{tikzpicture}[my angle/.style = {draw, angle radius=7mm, angle eccentricity=1.1, right, inner sep=1pt, font=\footnotesize}]
        \draw (0,0) coordinate[label=below:$\widetilde{x}^*$] (a) --
        (2.7,0) coordinate[label=below:$\widetilde{z}$] (c) --
        (4,1.8) coordinate[label=above:$\widetilde{x}_j$] (b) -- cycle;
        \pic[my angle, "$\widetilde{\gamma}_j$"] {angle = c--a--b};

        \node at (1.5, -0.4) {$> C_{\alpha} r$}; 
        \node at (3.8, 0.8) {$\leq r$};      

        \draw[dashed] (c) -- (2.15, 0.95); 
    \end{tikzpicture}\hfill
    \begin{tikzpicture}[my angle/.style = {draw, angle radius=7mm, angle eccentricity=1.1, right, inner sep=1pt, font=\footnotesize}]
        \draw (0,0) coordinate[label=below:$\widetilde{x}^*$] (a) --
        (4,0) coordinate[label=below:$\widetilde{z}$] (c) --
        (2.7,1.8) coordinate[label=above:$\widetilde{x}_j$] (b) -- cycle;
        \pic[my angle, "$\widetilde{\gamma}_j$"] {angle = c--a--b};

        \node at (2, -0.4) {$> C_{\alpha} r$}; 
        \node at (3.7, 1) {$\leq r$};      
    \end{tikzpicture}\hfill
    \begin{tikzpicture}[my angle/.style = {draw, angle radius=7mm, angle eccentricity=1.1, right, inner sep=1pt, font=\footnotesize}]
        \draw (0,0) coordinate[label=below:$\widetilde{x}^*$] (a) --
        (4,0) coordinate[label=below:$\widetilde{z}$] (c) --
        (3 , 1) coordinate[label=above:$\widetilde{x}_j$] (b) -- cycle;
        \pic[my angle, "$\widetilde{\gamma}_j$"] {angle = c--a--b};

        \node at (2, -0.4) {$> C_{\alpha} r$}; 
        \node at (3.2, 0.3) {$r \geq$};      
        \draw[dashed] (c) -- (3.5, 1.2); 
        \draw[dashed] (b) -- (3.5, 1.2); 
    \end{tikzpicture}
    \caption{Possible configurations of the triangle $\triangle \widetilde{x}^* \widetilde{x}_j \widetilde{z}$ for $j = 1, \dots \floor{(1-\alpha)k} + 1$. Clearly, the second case with the equality gives the tightest upper bound on $\sin(\widetilde{\gamma}_j)$, which is $1/C_{\alpha}$. This automatically yields the lower bound of the $\cos(\widetilde{\gamma}_j)$.}
    \label{figure_npc_triangle} 
\end{figure}

\textbf{CASE II: $\kappa > 0$.}

Since we assumed $x^*$ exists in this case, Equation \eqref{eq_minimizer_condi} holds. By the assumption, for any $x_j$, the triangle inequality implies
\[
    d(x^*, x_j) + d(x^*, z) + d(x_j, z) \leq 2 (d(x^*, x_j) + d(x^*, z) ) \leq 2 D_\kappa.
\]
Hence, $\{x^*, x_j, z\}$ can be embedded to $M^2_\kappa := \frac{1}{\sqrt{\kappa}} \mathbb{S}^2$. In other words, one can pick points $\widetilde{x}_j$'s, $\widetilde{x}^*$ and $\widetilde{z}$ on $M^2_\kappa$ such that any $\{\widetilde{x}^*, \widetilde{x}_j, \widetilde{z}\}$ forms the comparison triangle of $\{x^*, x_j, z\}$: see \citet{kunzinger2018alexndrov}[Proposition 4.2.20].

Applying the same method as in $\kappa \leq 0$ case, one can check
\[
    \limsup_{t \rightarrow 0} \frac{F(\alpha(t)) - F(\alpha(0))}{t} \leq -\sum_{j=1}^{k}  \cos(\widetilde{\gamma}_j)\mathds{1}_{\{\widetilde{x}_{j} \neq \widetilde{x}^*\}} + \sum_{j=1}^{k} \mathds{1}_{\{\widetilde{x}_{j} = \widetilde{x}^*\}}
\]
where $\widetilde{\gamma}_j$ is the angle at $\widetilde{x}^*$ of $\triangle \widetilde{x}^* \widetilde{x}_j \widetilde{z}$ on $M^2_\kappa$. 
Now, we apply Rauch comparison theorem on $M_{\kappa}^2$ \citep{kinzinger2014theexponential}[Theorem 3.1]. Denoting the intrinsic metric of $M_{\kappa}^2$ as $\widetilde{d}$, we obtain
\[
\frac{\sin \left(\sqrt{\kappa}R\right)}{\sqrt{\kappa}R} \norm{u} \leq \norm{d_{v}\exp_{\widetilde{x}^*}(u)}
\]
for any $\norm{v} \leq R < D_{\kappa}$. In fact, since $M_{\kappa}^2$ is $\bbS^{2}/\sqrt{\kappa}$, we can plug-in $R = D_{\kappa}/2$, the half of injectivity radius of $\bbS^{2}/\sqrt{\kappa}$. Under this choice of $R$, $B(\widetilde{x}^*, D_{\kappa}/2) \subset M_{\kappa}^2$ becomes convex. Therefore, the above estimate becomes global in such ball, so that
\[
\frac{2}{\pi} \norm{u - v}_{\widetilde{x}^*} \leq \widetilde{d}\left(\exp_{\widetilde{x}^*}(u), \exp_{\widetilde{x}^*}(v)\right).
\]
See \citet{fefferman2019reconstructioninterpolationmanifoldsi}[Equation (4.1), (4.2)] for more detail of this procedure. This results in the logarithmic map $\log_{\widetilde{x}^*}\cdot$ being $\pi/2$-Lipschitz in the ball $B(\widetilde{x}^*, D_{\kappa}/2)$. Since $\widetilde{x}_j, \widetilde{z} \in B(\widetilde{x}^*, D_{\kappa}/2)$, we have 
\[
\norm{\log_{\widetilde{x}^*} \widetilde{z} - \log_{\widetilde{x}^*} \widetilde{x}_j}_{\widetilde{x}^*} \leq \frac{\pi}{2} \widetilde{d}(\widetilde{z}, \widetilde{x}_j) \leq \frac{\pi}{2}r
\]
for $j = 1, \dots, \floor{(1-\alpha)k} + 1$. Therefore, by the considering the same Figure \ref{figure_npc_triangle} for the triangle constructed by $0, \log_{\widetilde{x}^*} \widetilde{x}_j$, and $\log_{\widetilde{x}^*} \widetilde{z}$ in the tangent space $T_{\widetilde{x}^*}M_{\kappa}^2$, $\sin(\widetilde{\gamma}_j)$ is upper bounded by $1/C_{\alpha}$ for $j = 1, \dots, \floor{(1-\alpha)k} + 1$ again. Therefore, the same procedure as in Case I yields
\[
    0 \leq \limsup_{t \rightarrow 0} \frac{F(\alpha(t)) - F(\alpha(0))}{t} \leq -\sum_{j=1}^{k}  \cos(\widetilde{\gamma}_j)\mathds{1}_{\{\widetilde{x}_{j} \neq \widetilde{x}^*\}} + \sum_{j=1}^{k} \mathds{1}_{\{\widetilde{x}_{j} = \widetilde{x}^*\}} < -(1-\alpha) k \sqrt{1 - \frac{1}{ C^2_{\alpha}}} + \alpha k \leq 0
\]
whenever $C_{\alpha} \geq (1-\alpha)\sqrt{\frac{1}{1-2\alpha}}$. This contradicts to Equation \eqref{eq_minimizer_condi}.
\end{proof}

\textbf{Proof of Theorem \ref{thm_concen_of_med}:}
\begin{proof}

For $\kappa \leq 0$, let $\calE:= \{d(\widehat{\theta}_{MoE}, \theta) > C_{\alpha} \epsilon\}$ the event. Then, under this event, by Lemma \ref{lem_geo_med} with $x^* = \widehat{\theta}_{MoE}$, $x_j = \widehat{\theta}_j$, and $z = \theta$, we have $J \subseteq \{1, \dots, k\}$ such that $|J| > \alpha k$ and $d(\widehat{\theta}_j, \theta) > \epsilon$ for $j \in J$. Let $W \sim B(k,p)$ be Binomial random variable. Then, 
\begin{align*}
    \bbP(\calE) \leq \bbP\left(\sum_{j=1}^{k} \mathds{1}_{\{d(\widehat{\theta}_j, \theta) > \epsilon\}} \geq \alpha k \right) \leq \bbP(W \geq \alpha k) \leq \exp\left(-k (1-\alpha) \log \frac{1-\alpha}{1-p} -k \alpha \log \frac{\alpha}{p}\right)
\end{align*}
where the second and the last inequalities follow from \citet{lerasle2011robust}[Lemma 23], and Chernoff bound respectively. 

The case for $\kappa > 0$ goes almost similarly, once one sets the event as $\calE:=\set{\pi C_{\alpha} \epsilon/2 < d(\widehat{\theta}_{FMoE}, \theta) \leq D_{\kappa}/2}$. The upper bound $d(\widehat{\theta}_{FMoE}, \theta) \leq D_{\kappa}/2$ vanishes by the assumption.
\end{proof}

Lastly, for the conditional probability in Remark \ref{remark_boosting_k_positive}, one can check
\[
d(\widehat{\theta}_{FMoE}, \theta) \leq \min_{j} \left[d(\widehat{\theta}_{FMoE}, \widehat{\theta}_j) + d(\widehat{\theta}_j, \theta)\right] \leq \frac{D_{\kappa}}{2} - \epsilon + \min_j d(\widehat{\theta}_j, \theta)
\]
almost surely. Now, from the assumption, 
\[
\bbP\left[\min_j d(\widehat{\theta}_j, \theta) \leq \epsilon \right] \geq 1 - p^{k}.
\]
Therefore, 
\[
\bbP \left[d(\widehat{\theta}_{FMoE}, \theta) \leq \frac{D_{\kappa}}{2}\right] \geq 1 - p^k.
\]
Then, the definition of the conditional probability leads to the claimed bound.

\subsection{Proofs in Section \ref{section_applications}} \label{appendix_proof_section_applications}

\textbf{Proof of Theorem \ref{thm_empirical_mean_mom}:}
\begin{proof}
From Proposition \ref{prop_expected_error_mean_estimators}, we obtain
    \[
    \bbE\left[d^2(\widehat{x}, x^*)\right] \leq \frac{\sigma^2}{\floor{n/k}} \leq \frac{2k\sigma^2}{n}.
    \]
Fixing $\alpha \in (0,0.5)$ and $p \in (0,\alpha)$, and then define $k := \floor{\log(1/\delta) / \psi(\alpha, p)} + 1$ and $\epsilon:= \sqrt{(2k\sigma^2) / (np)}$. Applying Markov inequality, one obtains
    \[
    \bbP\left[d(\widehat{x}, x^*) \geq \epsilon \right] \leq \frac{\bbE\left[d^2(\widehat{x}, x^*)\right]}{\epsilon^2} = \frac{2k\sigma^2}{n \epsilon^2} = p.
    \]
    Then, it follows from Theorem \ref{thm_concen_of_med} that 
    \[
    \bbP\left[d(\widehat{x}_{FMoM}, x^*) \geq C_{\alpha} \epsilon \right] \leq \exp\left(-k \psi(\alpha, p)\right) \leq \delta
    \]
by the construction of $k$. Observe that
    \[
    C_{\alpha} \epsilon = C_{\alpha} \sqrt{\frac{2k \sigma^2}{np}} = C_{\alpha} \sqrt{\frac{2 \sigma^2}{n p \psi(\alpha, p)}} \sqrt{k \psi(\alpha, p)} \leq C_{\alpha} \sqrt{\frac{2 \sigma^2}{n p \psi(\alpha, p)}} \sqrt{\psi(\alpha, p) + \log(1/\delta)}.
    \]
Since $\alpha$ and $p$ are arbitrary, any choice of $\alpha$ and $p$ will induce the bound. In fact, by minimizing the right-most above term with respect to $\alpha$ and $p$, we obtain the tighter bound. In this case, we plugged-in $p = 1/10$ and $\alpha = 7/18$, as suggested in \citet{minsker2015geometric}[Corollary 4.1].
\end{proof}

\textbf{Proof of Proposition \ref{prop_tail_of_cov_matrix}:}
\begin{proof}

\textbf{Case I: $d = d_{AI}$.}

We write $\Delta := \Sigma^{-1/2} \widehat{\Sigma} \Sigma^{-1/2} - I_d = \Sigma^{-1/2}(\widehat{\Sigma} - \Sigma)\Sigma^{-1/2}$, and let $\widetilde{\lambda}, \widetilde{\sigma}$ be the eigenvalues and singular values of $\Delta$. Note that the eigenvalues of $\Sigma^{-1/2} \widehat{\Sigma} \Sigma^{-1/2}$ are $\widetilde{\lambda} + 1$. 

Recall the inequalities between the spectral norm and the Frobenius norm:
\begin{equation}\label{eq: 2norm vs F_norm}
    \norm{\cdot}_{2} \leq \norm{\cdot}_{F} \leq \sqrt{d}\norm{\cdot}_{2}.
\end{equation}
Here, $\norm{\cdot}_F, \norm{\cdot}_{2}$ stand for Frobenius norm and spectral norm respectively. Equation \eqref{eq: 2norm vs F_norm} implies
\[
    \bbP\left[d_{AI}\left(\Sigma, \widehat{\Sigma}\right) \geq \epsilon\right] = \bbP\left[\norm{\log \Sigma^{-1/2} \widehat{\Sigma} \Sigma^{-1/2}}_{F} \geq \epsilon \right] \leq \bbP\left[\sqrt{d}\norm{\log \Sigma^{-1/2} \widehat{\Sigma} \Sigma^{-1/2}}_{2} \geq \epsilon \right].
\]
Note that if $\lambda$ is the eigenvalue of a positive semi-definite matrix $A$, then $\log \lambda$ is that of $\log A$. Hence,
\begin{align*}
    \bbP\left[\sqrt{d}\norm{\log \Sigma^{-1/2} \widehat{\Sigma} \Sigma^{-1/2}}_{2} \geq \epsilon \right] &= \bbP\left[ \max_{j \leq d} \abs{\log(\widetilde{\lambda}_j + 1)} \geq \frac{\epsilon}{\sqrt{d}} \right] \notag \\
    &= \bbP \left[\left\{\max_{j \leq d} \log (\widetilde{\lambda}_j + 1) \geq \frac{\epsilon}{\sqrt{d}} \right\} \cup \left\{\min_{j \leq d} \log (\widetilde{\lambda}_j + 1) \leq -\frac{\epsilon}{\sqrt{d}} \right\} \right] \notag\\
    &\leq \bbP \left[ \widetilde{\lambda}_{\max} \geq e^{\frac{\epsilon}{\sqrt{d}}} - 1 \right] + \bbP\left[ \widetilde{\lambda}_{\min} \leq e^{-\frac{\epsilon}{\sqrt{d}}} -1 \right] \notag\\
    &\leq \bbP\left[\widetilde{\sigma}_{\max} \geq 1 - e^{-\frac{\epsilon}{\sqrt{d}}}\right]
\end{align*}
where the last inequality follows from the fact that $\cosh(x) \geq 0$ and singular value bounds the absolute value of eigenvalues. Observe the submultiplicativity of singular values
\begin{equation}\label{eq: submulti singular value}
    \sigma_{\max}(ABC) \leq \sigma_{\max}(A) \sigma_{\max}(B) \sigma_{\max}(C).
\end{equation}
Taking $A, C = \Sigma^{-1/2}$ and $B = \widehat{\Sigma} - \Sigma$, Equation \eqref{eq: submulti singular value} leads to
\begin{align*}
    \bbP\left[\widetilde{\sigma}_{\max} \geq 1 - e^{-\frac{\epsilon}{\sqrt{d}}}\right] &\leq \bbP\left[\lambda_{\max}(\Sigma^{-1/2})^2 \lambda_{\max}\left(\abs{\widehat{\Sigma} - \Sigma}\right) \geq 1 - e^{-\frac{\epsilon}{\sqrt{d}}}\right]\\
    &= \bbP\left[\lambda_{\max}\left(\abs{\widehat{\Sigma} - \Sigma}\right) \geq \lambda_{\min}\left(1 - e^{-\frac{\epsilon}{\sqrt{d}}}\right)\right]\\
    &= \bbP \left[\norm{\widehat{\Sigma} - \Sigma}_{2} \geq \lambda_{\min}\left(1 - e^{-\frac{\epsilon}{\sqrt{d}}}\right)\right]\\
    &\leq \bbP\left[\norm{\widehat{\Sigma} - \Sigma}_{F} \geq \lambda_{\min}\left(1 - e^{-\frac{\epsilon}{\sqrt{d}}}\right)\right]
\end{align*}
where, again, the last inequality follows from Equation \eqref{eq: 2norm vs F_norm}. Combining the above argument, 
\begin{equation}\label{eq: SPD eq1}
    \bbP\left[d_{AI}\left(\Sigma, \widehat{\Sigma}\right) \geq \epsilon\right] \leq \bbP\left[\norm{\widehat{\Sigma} - \Sigma}_{F} \geq \lambda_{\min} \left(1 - e^{-\frac{\epsilon}{\sqrt{d}}}\right)\right].
\end{equation}

Now, writing $X_i = (X_i^1, \dots, X_i^d) \in \bbR^d$, it follows that
    \begin{align*}
    \bbP\left[\norm{\widehat{\Sigma} - \Sigma}_{F} \geq \lambda_{\min}\left(1 - e^{-\frac{\epsilon}{\sqrt{d}}}\right)\right] &= \bbP\left[\sum_{k, l=1}^{d} \abs{\frac{1}{n} \sum_{i = 1}^{n} X_{i}^k X_{i}^l - \bbE(X^k X^l)}^2 \geq \lambda_{\min}^2\left(1 - e^{-\frac{\epsilon}{\sqrt{d}}}\right)^2\right]\\
    &\leq \sum_{k,l = 1}^{d} \bbP\left[\abs{\frac{1}{n} \sum_{i = 1}^{n} X_{i}^k X_{i}^l - \bbE(X^k X^l)}^2 \geq \frac{\lambda_{\min}^2}{d^2}\left(1 - e^{-\frac{\epsilon}{\sqrt{d}}}\right)^2\right]\\
    &\leq \sum_{k,l=1}^{d} \frac{d^2 \bbE\abs{\frac{1}{n} \sum_{i = 1}^{n} X_{i}^k X_{i}^l - \bbE(X^k X^l)}^2}{\lambda_{\min}^2\left(1 - \exp\left(-\frac{\epsilon}{\sqrt{d}}\right)\right)^2}\\
    &= \sum_{k,l=1}^{d} \frac{d^2  \sum_{i,j = 1}^{n} \bbE \left( X_{i}^k X_{i}^l X_{j}^k X_{j}^l \right) - [\bbE(X^k X^l)]^2}{n^2\lambda_{\min}^2 \left(1 - \exp\left(-\frac{\epsilon}{\sqrt{d}}\right)\right)^2}
    \end{align*}
where thee second last inequality above follows by Markov inequality. Since $X_i$ and $X_j$ are independent,
\begin{align*}
    \sum_{k,l=1}^{d} \frac{d^2  \sum_{i,j = 1}^{n} \bbE \left( X_{i}^k X_{i}^l X_{j}^k X_{j}^l \right) - [\bbE(X^k X^l)]^2}{n^2\lambda_{\min}^2 \left(1 - \exp\left(-\frac{\epsilon}{\sqrt{d}}\right)\right)^2} &= \sum_{k,l=1}^{d} \frac{d^2  \sum_{i = 1}^{n} \bbE [\left( X_{i}^k X_{i}^l\right)^2] - [\bbE(X^k X^l)]^2}{n^2\lambda_{\min}^2\left(1 - \exp\left(-\frac{\epsilon}{\sqrt{d}}\right)\right)^2}.
\end{align*}
Therefore,
\begin{align*}
    \bbP\left[\norm{\widehat{\Sigma} - \Sigma}_{F} \geq \lambda_{\min}(\Sigma)\left(1 - e^{-\frac{\epsilon}{\sqrt{d}}}\right)\right] &\leq \sum_{k,l=1}^{d} \frac{d^2  \sum_{i = 1}^{n} \bbE [\left( X_{i}^k X_{i}^l\right)^2] - [\bbE(X^k X^l)]^2}{n^2\lambda_{\min}^2 \left(1 - \exp\left(-\frac{\epsilon}{\sqrt{d}}\right)\right)^2}\\
    &= \frac{d^2}{n \lambda_{\min}^2\left(1 - \exp\left(-\frac{\epsilon}{\sqrt{d}}\right)\right)^2} \sum_{k,l=1}^{d} Var(X^k X^l).
\end{align*}
Since we assumed $P \in \calP_4(\bbR^d)$, we have $\sum_{k,l=1}^{d} Var(X^k X^l) \leq C d^2$ for some constant $C > 0$. The conclusion follows.

\textbf{Case II: $d = d_{BW}$.}

Using the fact that eigenvalues of $\widehat{\Sigma}, \Sigma$ are lower bounded by $\lambda_0$, the following estimate for Bures-Wasserstein distance holds:
\[
d_{BW}^2(\widehat{\Sigma}, \Sigma) \overset{\text{(i)}}{=} \min_{U: \text{Unitary}}\norm{\widehat{\Sigma}^{1/2} - \Sigma^{1/2}U}_F^2 \leq \norm{\widehat{\Sigma}^{1/2} - \Sigma^{1/2}}_F^2 \overset{\text{(ii)}}{\leq} \frac{1}{4 \lambda_0} \norm{\widehat{\Sigma} - \Sigma}_{F}^2
\]
where (i) and (ii) are from \citet{bhatia2019bureswasserstein}[Theorem 1] and \citet{ando1980inequality}[Proposition 3.2] respectively.

Hence, we get
\[
\bbP\left[d_{BW}\left(\widehat{\Sigma}, \Sigma \right) \geq \epsilon \right] \leq 
\bbP\left[\norm{\widehat{\Sigma} - \Sigma}_{F} \geq 2 \sqrt{\lambda_0} \epsilon \right] \leq \frac{C d^4}{4 n  \lambda_0 \epsilon^2}
\]
where the last inequality can be obtained similarly to the last step in Case I.
\end{proof}

\textbf{Proof of Theorem \ref{thm_exp_conc_cov_mat}:}
\begin{proof}
    We follow the same technique as in Theorem \ref{thm_empirical_mean_mom}. We fix $\alpha$ and $p$ and then set $k = \floor{\log(1/\delta)/\psi(\alpha, p)} + 1$. 

    \textbf{Case I: $d = d_{AI}$.}
    
    From Proposition \ref{prop_tail_of_cov_matrix}:
    \[
    \bbP\left[d_{AI}\left(\widehat{\Sigma}, \Sigma \right) \geq \epsilon \right] \leq \frac{2 k C d^4}{n \lambda_{\min}^2 \left(1 - \exp\left(-\frac{\epsilon}{\sqrt{d}}\right)\right)^2}:= p
    \]
    by choosing 
    \[
    \epsilon = -\sqrt{d} \log \left(1 - \frac{\sqrt{2 k C} d^2}{\lambda_{\min}\sqrt{np}}\right)
    \]
    which is always possible from the condition $n > 2 k C d^4 / \lambda_{\min}^2$.
    Then, we apply Theorem \ref{thm_concen_of_med} to yield
    \[
    \bbP\left[d_{AI}\left(\widehat{\Sigma}_{FMoE}, \Sigma\right) \geq C_{\alpha} \epsilon\right] \leq \exp \left(- k \psi(\alpha, p)\right) \leq \delta
    \]
    by the setting of $k$. Observe
    \begin{align*}
        C_{\alpha} \epsilon &= -\sqrt{d} C_{\alpha} \log \left(1 - \frac{\sqrt{2 k C} d^2}{\lambda_{\min}\sqrt{np}}\right) \leq -\sqrt{d} C_{\alpha} \log \left(1 - \frac{\sqrt{2 C} d^2}{ \lambda_{\min}\sqrt{n p \psi(\alpha, p)}} \sqrt{\psi(\alpha, p) + \log(1/\delta)} \right).
    \end{align*}
    Again, this bound holds for all $\alpha \in (0,0.5)$ and $p \in (0, \alpha)$. Plugging-in $\alpha = 0.4$ and $p = 0.1$ yields the claimed result.

    \textbf{Case II: $d = d_{BW}$.} 

    For $d_{BW}$, one can employ the similar technique as in the above with $\epsilon = d^2 \sqrt{(k C)/(2 n \lambda_0 p)}$. We used the same $\alpha = 0.4$ and $p = 0.1$. The condition $n > 6 \lambda_0 k C d^4$ was imposed to guarantee $\epsilon < D_{\kappa}/(\pi C_{\alpha})$ for such choice of $\alpha$ and $p$.
\end{proof}

\textbf{The precise statement of Remark \ref{rmk_eigenvalue_lower_bound_concentration}:}



Notice that
\[
    | \lambda_{\min}(\widehat{\Sigma}) - \lambda_{\min}(\Sigma) | \leq \| \widehat{\Sigma} - \Sigma\|_2
\]
from the Weyl's inequality. Then, again applying Equation \eqref{eq: 2norm vs F_norm} and the calculation on the Frobenius norm bound in the proof of Proposition \ref{prop_tail_of_cov_matrix} lead to
\[
\bbP\left[\abs{\lambda_{\min}(\widehat{\Sigma}) - \lambda_{\min}(\Sigma)} \geq \epsilon \right] \leq \bbP\left[\norm{\widehat{\Sigma} - \Sigma}_{2} \geq \epsilon \right] \leq \bbP\left[\norm{\widehat{\Sigma} - \Sigma}_{F} \geq \epsilon \right] \leq \frac{C d^4}{n \epsilon^2}.
\]
Here, $C$ is the same $C$ in Proposition \ref{prop_tail_of_cov_matrix}.
Writing $\delta := C d^4/(n \epsilon^2)$, one gets
\[
\bbP\left[\lambda_{\min}(\widehat{\Sigma}) \in B\left(\lambda_{\min}(\Sigma), d^2 \sqrt{\frac{C}{n \delta}}\right)\right] \geq 1- \delta.
\]
This implies $\lambda_{\min}(\widehat{\Sigma})$ concentrates in $B(\lambda_{\min}(\Sigma), C/\sqrt{n})$ for some universal $C > 0$ with high probability whenever $P$ allows the finite fourth moment. 

\section{Implementation detail and additional experiments}\label{appendix_additional_experiment}

This section includes implementation detail and more experiments of our algorithm under different settings. First, for all experiments in the main paper, we show how the performance of our method varies by the block size and the population distribution. In addition, for Fr\'echet mean estimation problem, we conduct the additional experiment on a Poincar\'e disk model to verify our method works in various domains. All new experiments in this Appendix were conducted 100 times, while we maintained with the results in Section \ref{section_implementation_experi} for the same one. We found this number of simulation to be sufficient to obtain the coherent results.

\subsection{Implementation detail} \label{appendix_implementation_detail}

For implementations, we used Python package \texttt{Geomstats} \citep{miolane2020geomstats, miolane2024geomstats} to model particular CAT($\kappa$) spaces and compute the geometric quantities like metrics and geodesics. For the Fr\'echet mean estimation problem in a NPC space, Fr\'echet mean and median were implemented using inductive mean and \citet{bacak2014computing}[Algorithm 4.3]. For the covariance estimation problem, Fr\'echet mean and median were computed using predefined functions in \texttt{Geomstats} (which are based on subgradient methods). All experiments were performed on a free version of Google Colab without any use of GPU. For each task, running 100 simulations did not take more than 5 minutes.

\subsection{Effects of the block sizes and population distributions}

We first analyze the effect of the block size and the population distributions. For 5-legs spider tree, we used same $n = 100$, and varied the block size from $k = 1$ to $100$. Note that $k = 1$ case coincides to the original estimator within $n$ samples. For the population distribution, we maintained to use $P = \text{Unif}(1, \dots, 5) \times \left((1-\alpha) \abs{N(1,1)} + \alpha \abs{N(100,1)}\right)$, while varying $\alpha$ to observe the effect of the tail of the population distribution. The results are summarized in Table \ref{table_experi_appendix_5_leg_error} and \ref{table_experi_appendix_5_leg_ci}.

\begin{table}[!ht]
\caption{$\bbE d^2(\widehat{x}_{FMoE}, x^*)$ from 100 simulations in 5-legs spider with $\alpha$-portion outliers.} \label{table_experi_appendix_5_leg_error}
\begin{center}
\begin{tabular}{c|ccccc}
$\alpha$ & $k = 100$ & $k = 50$ &$k = 10$ & $k = 5$ & $k = 1$\\
\hline 
0 & $9.4 \times 10^{-9}$ & $3.4 \times 10^{-8}$ & $1.1 \times 10^{-5}$ & 0.0002 & 0.0003\\
0.1 & $8.6 \times 10^{-9}$ & $3.1 \times 10^{-8}$ & $1.2 \times 10^{-5}$ & 0.0105 & 3.1244\\
0.5 & $1.1 \times 10^{-8}$ & $4.3 \times 10^{-8}$ & 0.0223 & 0.0313 & 3.4634\\
0.9 & 0.0118 & 0.3560 & 0.0127 & 0.0193 & 3.8080
\end{tabular}
\end{center}
\end{table}

\begin{table}[!ht]
\caption{95\% confidence interval of $d(\widehat{x}_{FMoE}, x^*)$ from 100 simulations in 5-legs spider with $\alpha$-portion outliers.} \label{table_experi_appendix_5_leg_ci}
\begin{center}
\begin{tabular}{c|ccccc}
$\alpha$ & $k = 100$ & $k = 50$ &$k = 10$ & $k = 5$ & $k = 1$\\
\hline 
0 & $[2.6 \times 10^{-5}, 0.0001]$ & $[1.5 \times 10^{-5}, 0.0002]$ & $[6.9 \times 10^{-5}, 0.0127]$ & $[7.1 \times 10^{-5}, 0.0485]$ & $[0.0005, 0.0415]$\\
0.1 & $[8.2 \times 10^{-6}$, $0.0001]$ & $[1.1 \times 10^{-5}, 0.0002]$ & $[9.6 \times 10^{-5}, 0.0086]$ & $[7.3 \times 10^{-7}, 0.0711]$ & $[0.0110, 4.4202]$\\
0.5 & $[4.5 \times 10^{-5}, 0.0001]$ & $[0.0001, 0.0002]$ & $[0.0004, 0.2928]$ & $[5.5 \times 10^{-8}, 0.6888]$ & $[0.1158 4.5415]$\\
0.9 & $[8.1 \times 10^{-5}, 0.0004]$ & $[3.4 \times 10^{-5}, 2.5411]$ & $[0.0009, 0.1423]$ & $[7.3 \times 10^{-8}, 0.6190]$ & $[0.0566, 4.9039]$
\end{tabular}
\end{center}
\end{table}

For covariance estimation problem, we again used the same $d = 10$ and $n = 10d^4$, while varying $k = 1$ to $100$. We used the same procedure to generate the $\Sigma$, and then experimented on $t_{\nu}(0,
\Sigma)$ for different $\nu$ to observe the effect of the tail of the distribution. As $\nu \rightarrow 2$, distributions will impose a heavier tail. The results for covariance estimation problem are summarized in Table \ref{table_experi_appendix_cov_ai}, \ref{table_experi_appendix_cov_ai_ci}, \ref{table_experi_appendix_cov_bw}, and \ref{table_experi_appendix_cov_bw_ci}. 



In all experiments, when the tail is relatively light ($\alpha = 0$, $\nu = 5$), the original estimators ($k = 1$) are comparable to our proposed estimators. When the tail is very heavy ($\alpha = 0.9$, $\nu = 2.2$), interesting behaviors emerge. In these cases, the proper choices of the block size ($k = 5, 10$) yield higher accuracy as well as stronger concentration, which are in agreement with results in Section \ref{section_implementation_experi}. However, one can observe that performances worsen if the block sizes are set too large ($k = 50$ for spider, and  $k = 100$ for the covariance estimation). This observation aligns with the discussion in the beginning of Section \ref{section_implementation_experi}; if the block size is too large, the original estimator may not perform well within the subset size of $\floor{n/k}$, and our method may not work properly in such cases. 

\begin{table}[!ht]
\caption{$\bbE d_{AI}^2(\widehat{\Sigma}_{FMoE}, \Sigma)$ from 100 simulations for covariance estimation on $t_{\nu}(0,\Sigma)$.} \label{table_experi_appendix_cov_ai}
\begin{center}
\begin{tabular}{c|cccccc}
$\nu$ & $k = 100$ & $k = 10$ & $k = 5$ & $k = 1$\\
\hline 
2.2 & 6.5528 & 3.4604 & 2.9163 & 3.3515 \\
2.5 & 0.9739 & 0.3454 & 0.2931 & 0.6057\\
3 & 0.1327 & 0.0377 & 0.0387 & 0.1141\\
5 & 0.0055 & 0.0032 & 0.0032 & 0.0034\\
\end{tabular}
\end{center}
\end{table}

\begin{table}[!ht]
\caption{95\% confidence interval of $d_{AI}(\widehat{\Sigma}_{FMoE}, \Sigma)$ from 100 simulations for covariance estimation on $t_{\nu}(0,\Sigma)$.} \label{table_experi_appendix_cov_ai_ci}
\begin{center}
\begin{tabular}{c|cccccc}
$\nu$ & $k = 100$ & $k = 10$ & $k = 5$ & $k = 1$\\
\hline 
2.2 & $[1.2948, 2.3295]$ & $[1.6055, 2.0854]$ & $[1.3301, 1.9865]$ & $[1.2816, 3.3372]$ \\
2.5 & $[0.9035, 1.0604]$ & $[0.4239, 1.7378]$ & $[0.4132, 0.6792]$ & $[0.4266, 1.6865]$\\
3 & $[0.3027, 0.4070]$ & $[0.1422, 0.244011]$ & $[0.1485, 0.2562]$ & $[0.1488, 0.8103]$\\
5 & $[0.0527, 0.0934]$ & $[0.0445,  0.0666]$ & $[0.0447, 0.0686]$ & $[0.0444, 0.0756]$\\
\end{tabular}
\end{center}
\end{table}

\begin{table}[!ht]
\caption{$\bbE d_{BW}^2(\widehat{\Sigma}_{FMoE}, \Sigma)$ from 100 simulations for covariance estimation on $t_{\nu}(0,\Sigma)$.} \label{table_experi_appendix_cov_bw}
\begin{center}
\begin{tabular}{c|cccccc}
$\nu$ & $k = 100$ & $k = 10$ & $k = 5$ & $k = 1$\\
\hline 
2.2 & 61.0520 & 36.6296 & 31.5038 & 76.8935\\
2.5 & 4.6620 & 2.0343 & 1.7360  & 8.3281\\
3 & 0.3703 & 0.1407 & 0.1484 & 0.4993\\
5 & 0.0082 & 0.0065 & 0.0067 & 0.0074 \\
\end{tabular}
\end{center}
\end{table}

\begin{table}[!ht]
\caption{95\% confidence interval of $d_{BW}(\widehat{\Sigma}_{FMoE}, \Sigma)$ from 100 simulations for covariance estimation on $t_{\nu}(0,\Sigma)$.} \label{table_experi_appendix_cov_bw_ci}
\begin{center}
\begin{tabular}{c|cccccc}
$\nu$ & $k = 100$ & $k = 10$ & $k = 5$ & $k = 1$\\
\hline 
2.2 & $[6.5140, 8.8935]$& $[4.6619, 7.2466]$ & $[4.2740, 6.8831]$ & $ [4.0157, 25.6891]$ \\
2.5 & $[1.8019, 2.5497]$ & $[1.0648, 1.8412]$ & $[0.9443, 1.7409]$ & $[0.9936, 5.8796]$\\
3 & $[0.4492, 0.7285]$ & $[0.2632, 0.4909]$ & $[0.2631, 0.5110]$ & $[0.2777, 1.6146]$\\
5 & $[0.0680, 0.1222]$ & $[0.0600, 0.1078]$ & $[0.0584, 0.1053]$ & $[0.0578, 0.1126]$\\
\end{tabular}
\end{center}
\end{table}

\subsection{Fr\'echet mean estimation in Poincar\'e disk model}

Lastly, to verify our method works in various domains, we conduct the Fr\'echet mean estimation in Poincar\'e disk model, a widely used space for hierarchical model due to the pioneer work of \citet{nickel2017poincare}. Poincar\'e disk is a 2-dimensional Riemannian manifold with nonpositive sectional curvature (therefore a CAT($\kappa$) space) which can be embedded in the unit ball of $\bbR^2$. A Riemannian metric tensor of Poincar\'e disk is defined by the following formula:
\[
ds^2 = \frac{d x^2 + d y^2}{\left(1 - x^2 - y^2\right)^2}.
\]
A Poincar\'e disk can be embedded into an open Euclidean unit ball in $\bbR^2$. Specifically, we can construct a Poincar\'e disk using a method similar to the stereographic projection of a sphere. Consider the upper hyperboloid described by the equation $t^2 = x^2 + y^2 + 1$ for $ t > 1$. We can project this hyperboloid from the point $(t = -1, x= 0, y = 0)$ onto a unit disk at $t = 0$. This projection maps points on the hyperboloid to the unit disk, creating the Poincar\'e disk model. A distance between two points is given by the Euclidean length of the hyperbolic arc between corresponding points. Figure \ref{figure_poincare_ball_plot} illustrates the Poincar\'e disk. Intuitively, points closer to the boundary will have larger distances. 

\begin{figure*}[!ht]
\centering
\includegraphics[width=0.7\textwidth]{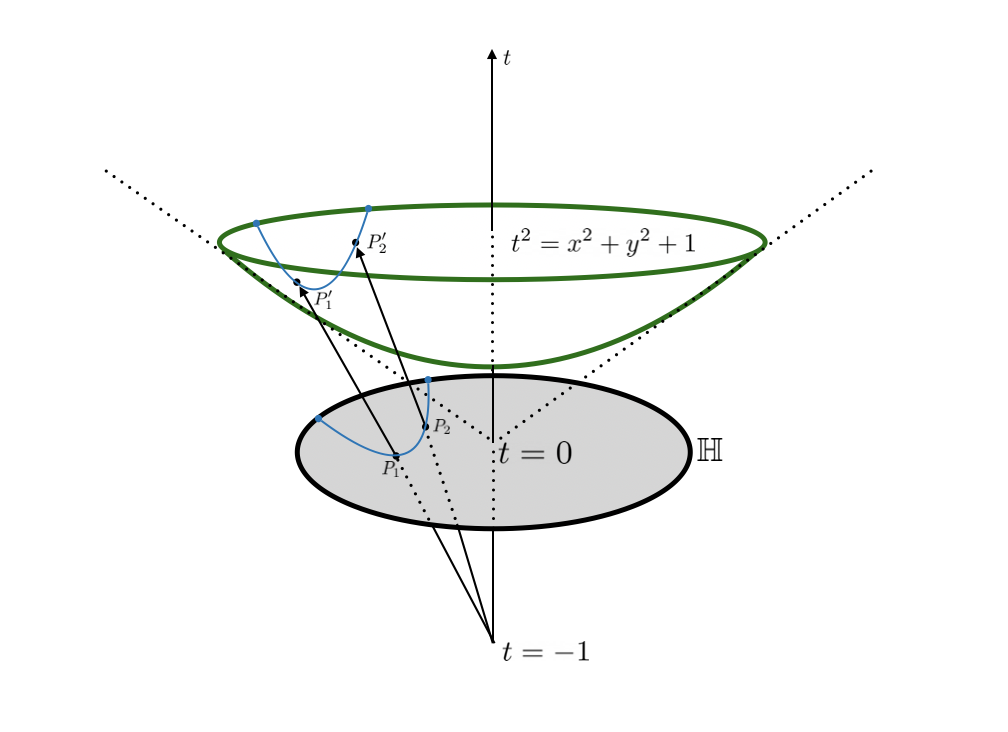}
\caption{The illustration of the Poincar\'e disk model. Poincar\'e disk model is defined by the intersection of the unit disk, e.g., $P_1, P_2$, and the projection map from the point $(t = -1, x = 0, y = 0)$ to the upper hyperboloid, e.g., $P_1', P_2'$. The distance between $P_1$ and $P_2$ is determined by the Euclidean length of the hyperbolic arc connecting their corresponding points $P_1'$ and $P_2'$.}\label{figure_poincare_ball_plot}
\end{figure*}

There are extensive theories regarding hyperbolic geometry and the Poincar\'e disk model. However, since these topics are not our primary interest, we refer interested readers to \citet{anderson2005hyperbolic}[Chapter 4.1]. Instead, we focus on numerically validating our method for estimating the Fréchet mean in the setting of heavy-tailed distributions in the Poincar\'e disk.

For the population distribution, we used the following mixture distribution in Poincare disk: 
\[
P = (1-\alpha) N(0, 0.2^2)\bigg\vert_{B(0,1)} + \alpha \text{Unif}\left((1-10^{-7})\bbS^1\right).
\]
Here, $N(0, 0.2^2)\bigg\vert_{B(0,1)}$ means the distribution of projected Gaussian random variable to the unit ball of Euclidean space. This distribution has Fr\'echet mean at the origin (due to its symmetricity), but it has a very heavy tail in Poincar\'e disk due to the effect of the outlier quantities, which are from the uniform distribution around the boundary. $\alpha$ again denotes the portion of outliers. We used the sample size $n = 100$, and experimented by changing the block size $k$ and the portion of outliers $\alpha$. The results are summarized in Table \ref{table_experi_appendix_poincare_error}, \ref{table_experi_appendix_poincare_ci}, and Figure \ref{figure_poincare_ball_ci}, \ref{figure_poincare_ball_res}.

\begin{table}[!ht]
\caption{$\bbE d^2(\widehat{x}_{FMoE}, x^*)$ from 100 simulations in Poincar\'e disk with $\alpha$-portion outliers.} \label{table_experi_appendix_poincare_error}
\begin{center}
\begin{tabular}{c|cccc}
$\alpha$ & $k = 50$ &$k = 10$ & $k = 5$ & $k = 1$\\
\hline 
0 & 0.0035 & 0.0047 & 0.0038 & 0.0028\\
0.1 & 0.0058 & 0.0323 & 0.1447 & 0.1063\\
0.5 & 0.0195 & 0.1553 & 0.1937 & 0.1709 \\
0.9 & 0.0667 & 0.1776 & 0.2320 & 0.1747 
\end{tabular}
\end{center}
\end{table}

\begin{table}[!ht]
\caption{95\% confidence interval of $d(\widehat{x}_{FMoE}, x^*)$ from 100 simulations in Poincar\'e disk with $\alpha$-portion outliers.} \label{table_experi_appendix_poincare_ci}
\begin{center}
\begin{tabular}{c|cccc}
$\alpha$ & $k = 50$ &$k = 10$ & $k = 5$ & $k = 1$\\
\hline 
0 & $[0.0157, 0.1032]$ & $[0.0081, 0.1358]$ & $[0.0111, 0.1198]$& $[0.0107, 0.0935]$\\
0.1 & $[0.0197, 0.1397]$ & $[0.0401, 0.4282]$ & $[0.0504, 0.7766]$ & $[0.0610, 0.6266]$\\
0.5 & $[0.0216, 0.2845]$ & $[0.0809, 0.7567]$ & $[0.0703, 0.8623]$ & $[0.0852,  0.7389]$\\
0.9 & $[0.0261, 0.5159]$ & $[0.0565, 0.7992]$ & $[0.1253, 0.8588]$ & $[0.0772,  0.8063]$
\end{tabular}
\end{center}
\end{table}


\begin{figure*}[!ht]
\centering
\includegraphics[width=0.33\textwidth]{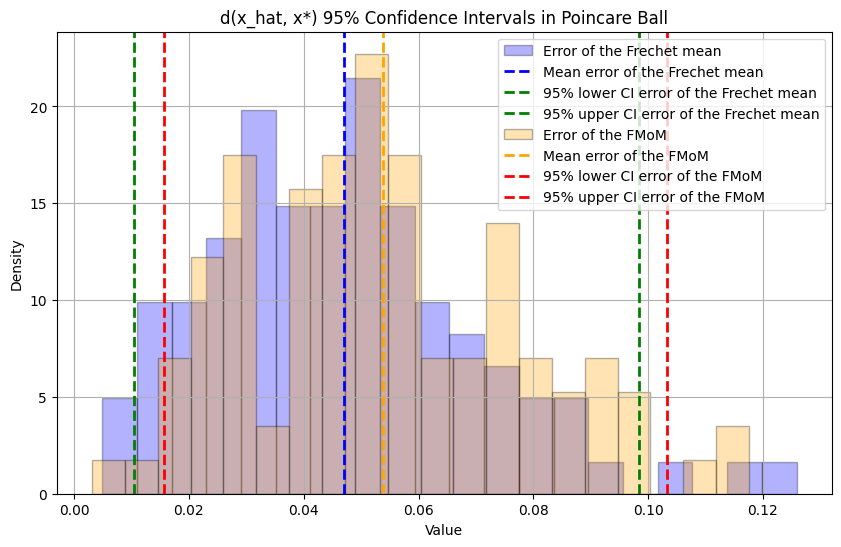}\hfill
\includegraphics[width=0.33\textwidth]{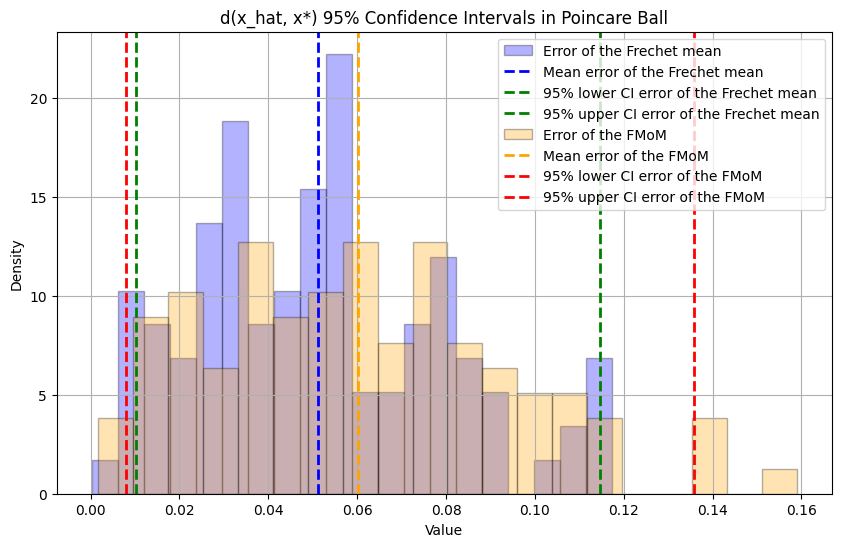}\hfill
\includegraphics[width=0.33\textwidth]{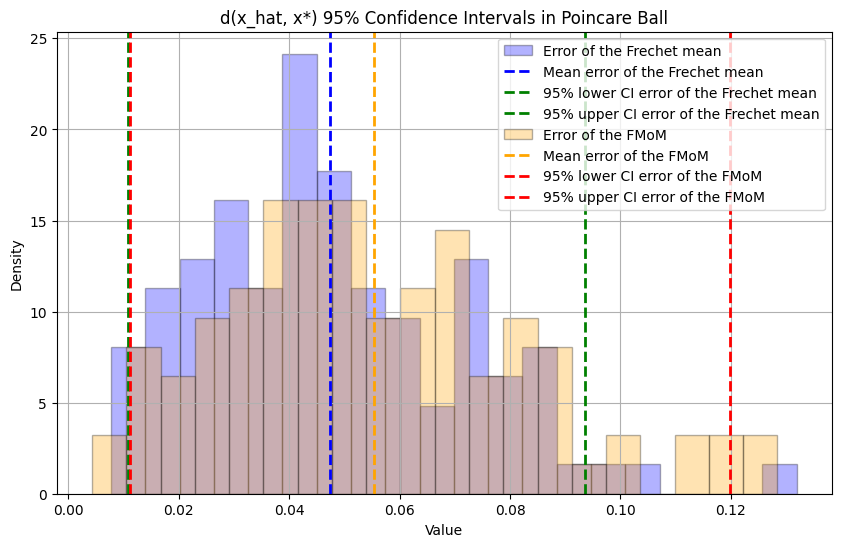}\\
\includegraphics[width=0.33\textwidth]{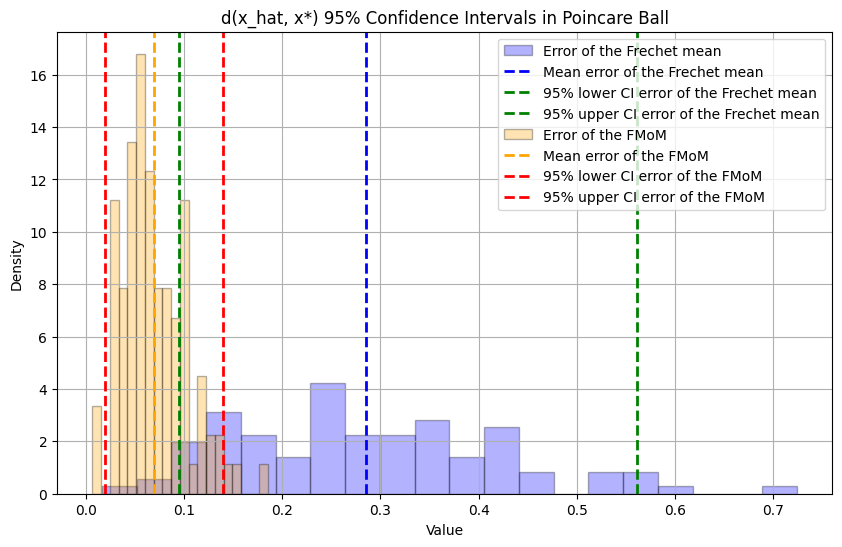}\hfill
\includegraphics[width=0.33\textwidth]{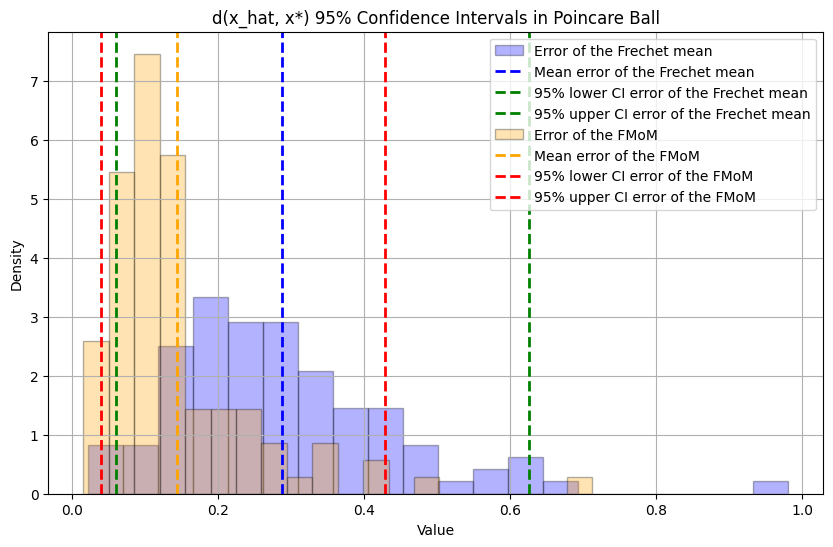}\hfill
\includegraphics[width=0.33\textwidth]{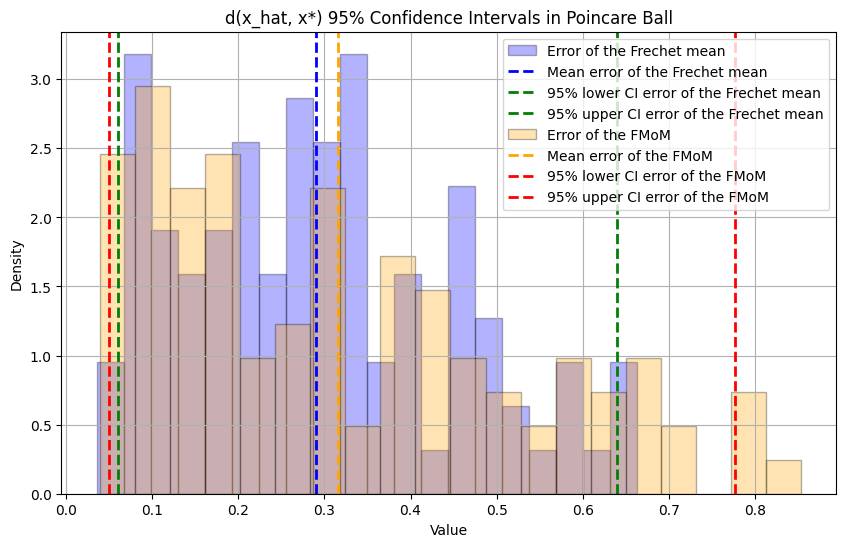}\\
\includegraphics[width=0.33\textwidth]{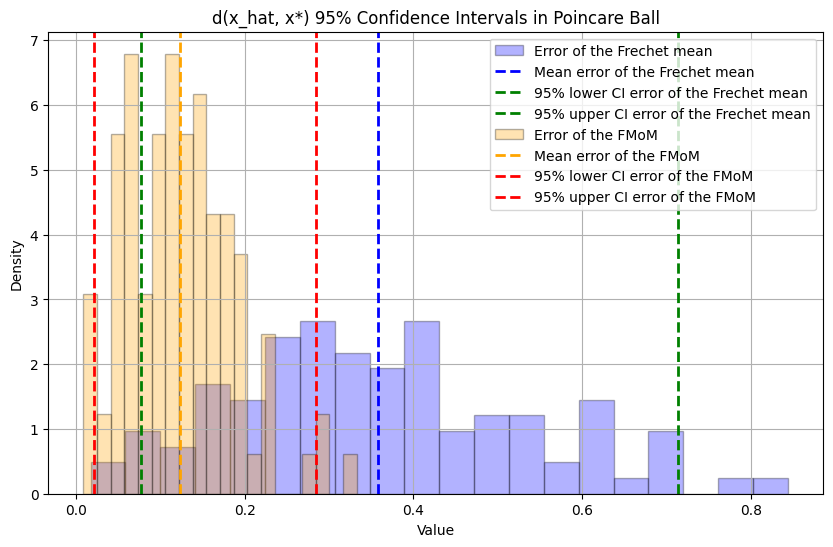}\hfill
\includegraphics[width=0.33\textwidth]{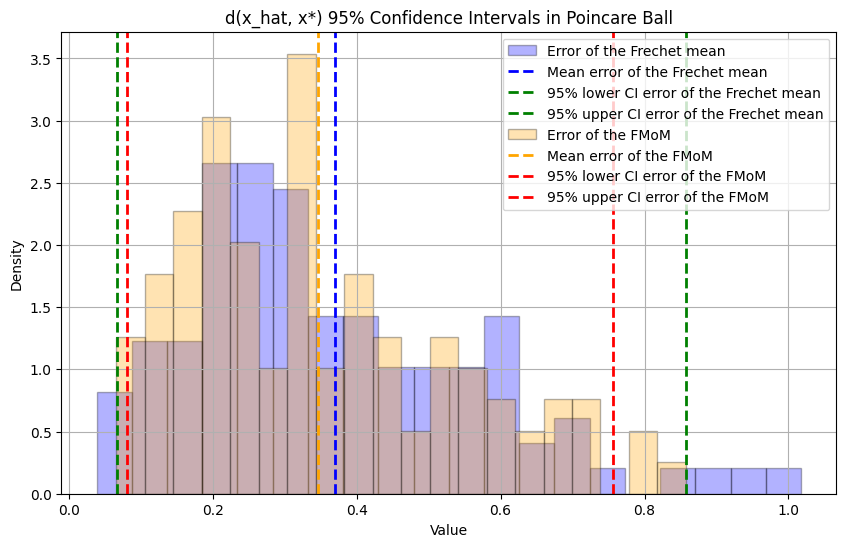}\hfill
\includegraphics[width=0.33\textwidth]{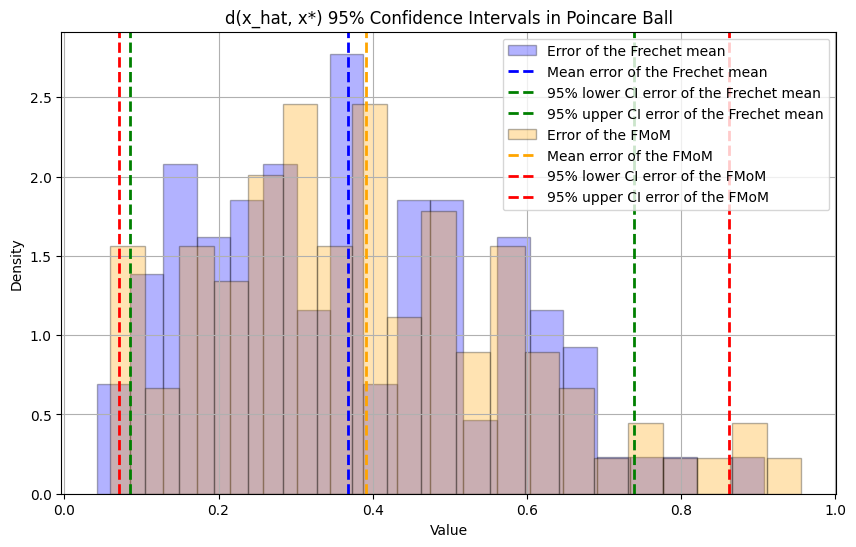}\\
\includegraphics[width=0.33\textwidth]{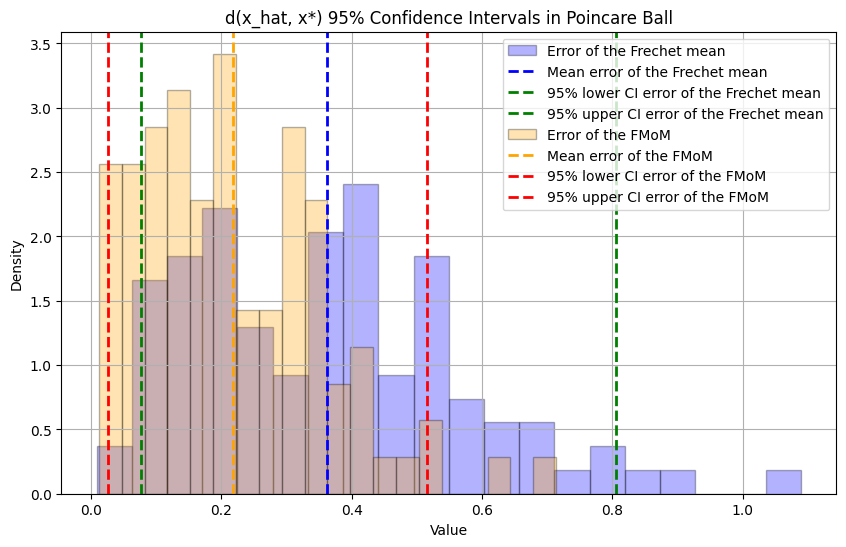}\hfill
\includegraphics[width=0.33\textwidth]{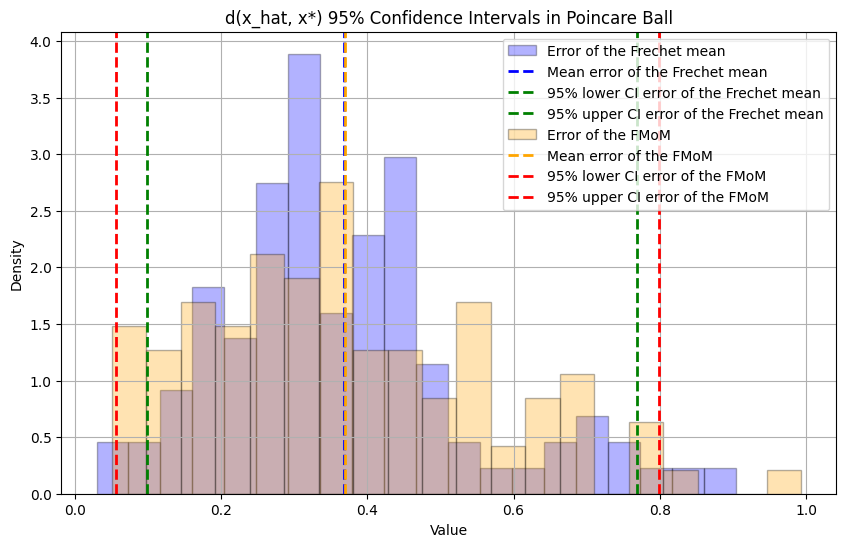}\hfill
\includegraphics[width=0.33\textwidth]{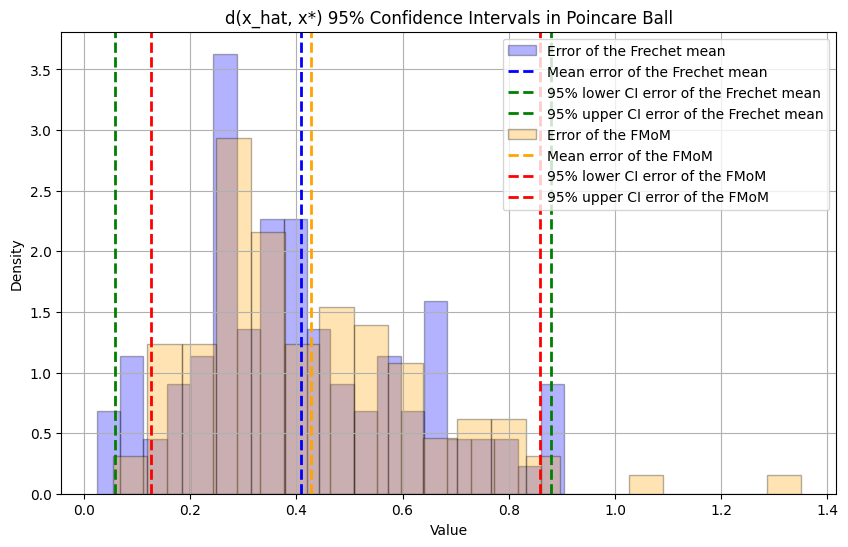}
\caption{Histogram, mean, and 95\% confidence interval for each experiment from 100 simulations. \textbf{Rows}: $\alpha = 0, 0.1, 0.5, 0.9$ from top to the bottom. \textbf{Columns}: $k = 50, 10, 5$ from left to right. For each experiment, comparison between the inductive mean estimator ($k = 1$) is displayed.}\label{figure_poincare_ball_ci}
\end{figure*}

\begin{figure*}[!ht]
\centering
\includegraphics[width=0.33\textwidth]{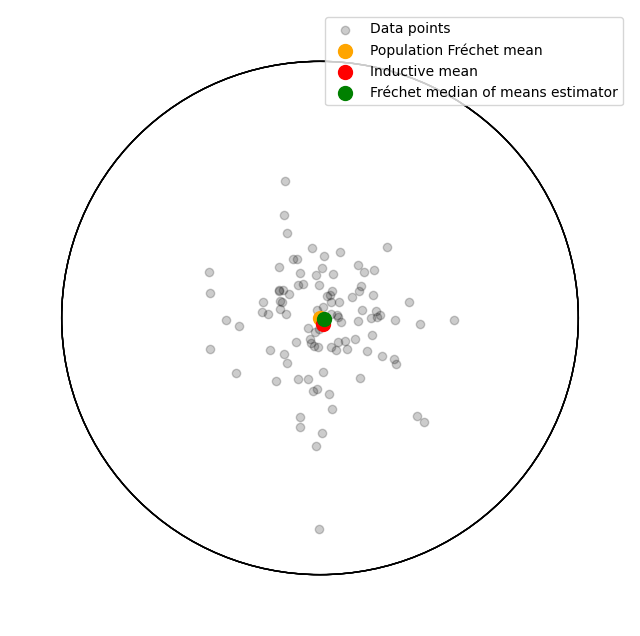}\hfill
\includegraphics[width=0.33\textwidth]{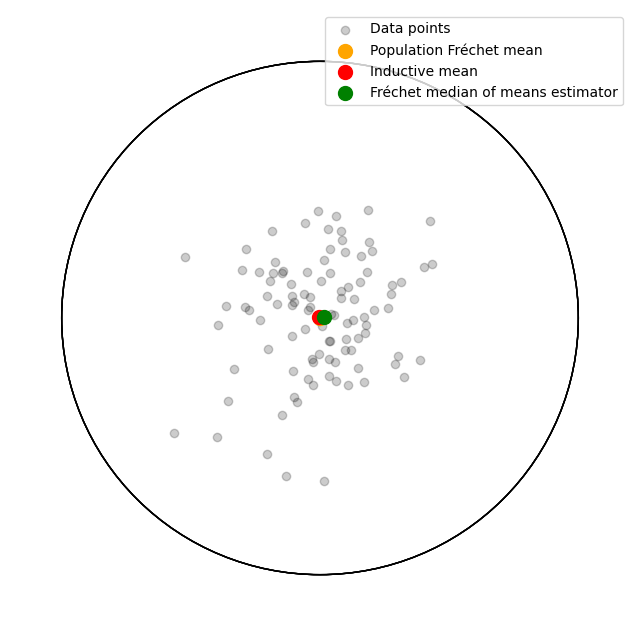}\hfill
\includegraphics[width=0.33\textwidth]{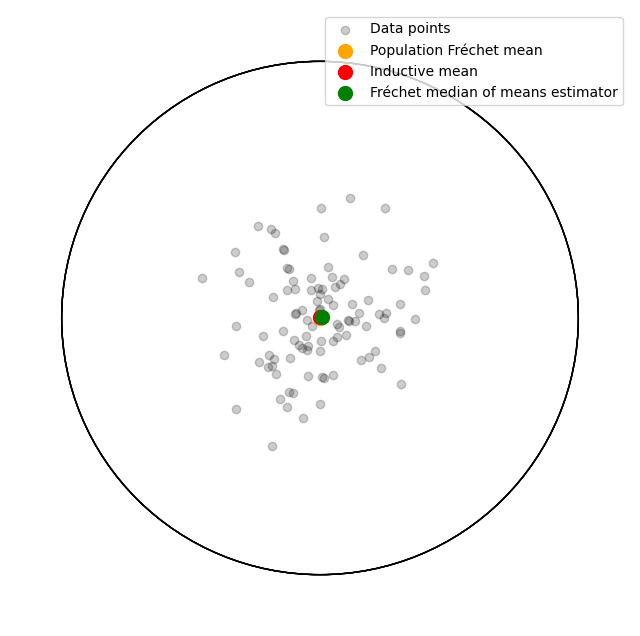}\\
\includegraphics[width=0.33\textwidth]{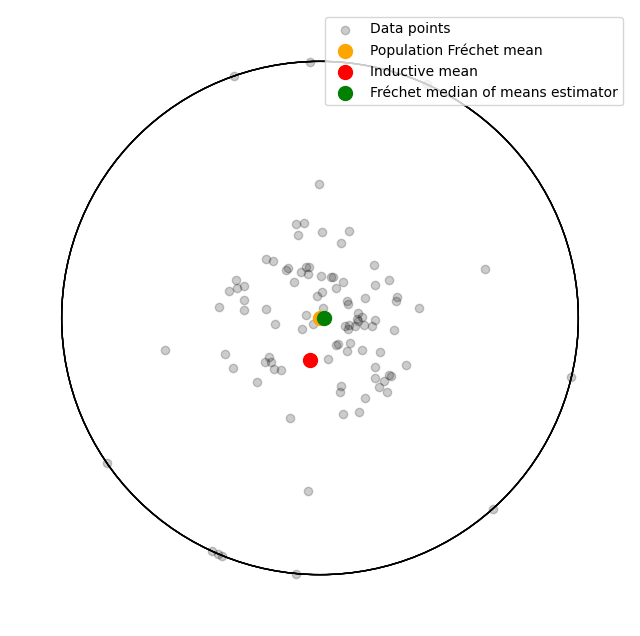}\hfill
\includegraphics[width=0.33\textwidth]{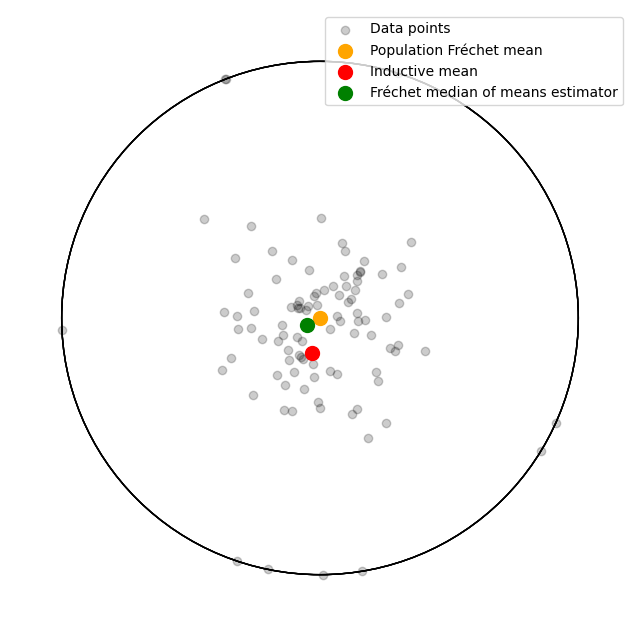}\hfill
\includegraphics[width=0.33\textwidth]{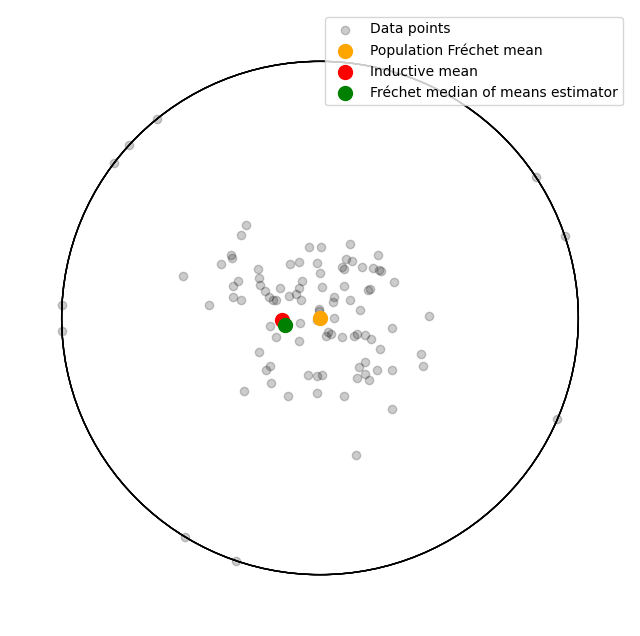}\\
\includegraphics[width=0.33\textwidth]{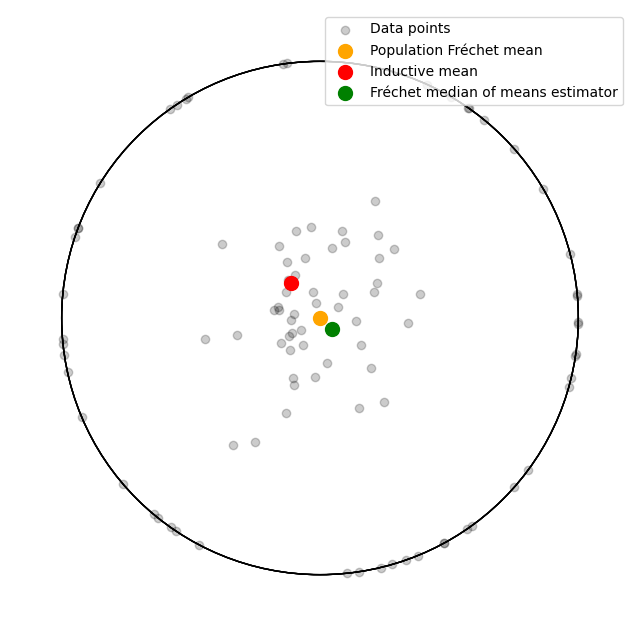}\hfill
\includegraphics[width=0.33\textwidth]{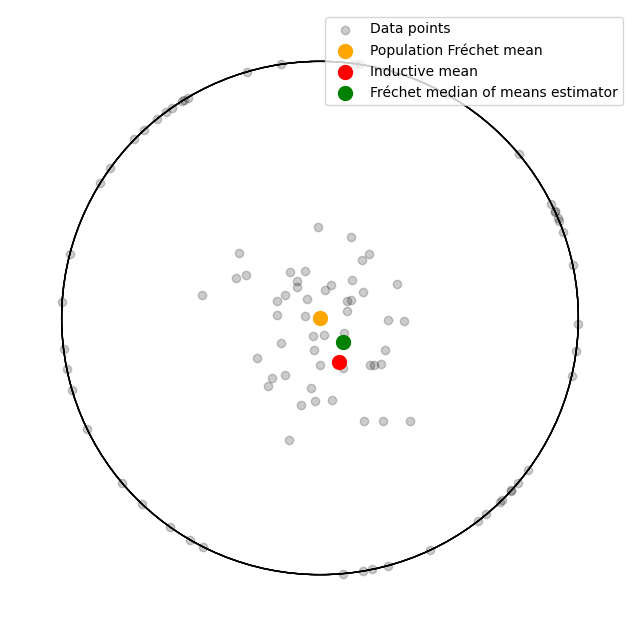}\hfill
\includegraphics[width=0.33\textwidth]{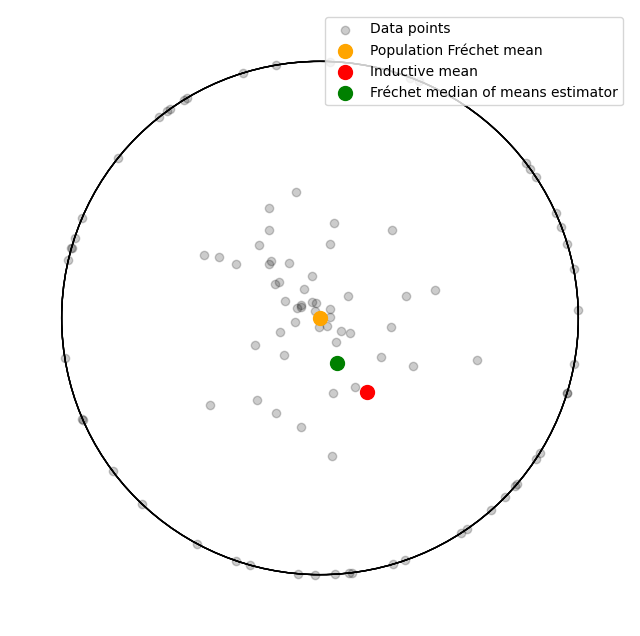}\\
\includegraphics[width=0.33\textwidth]{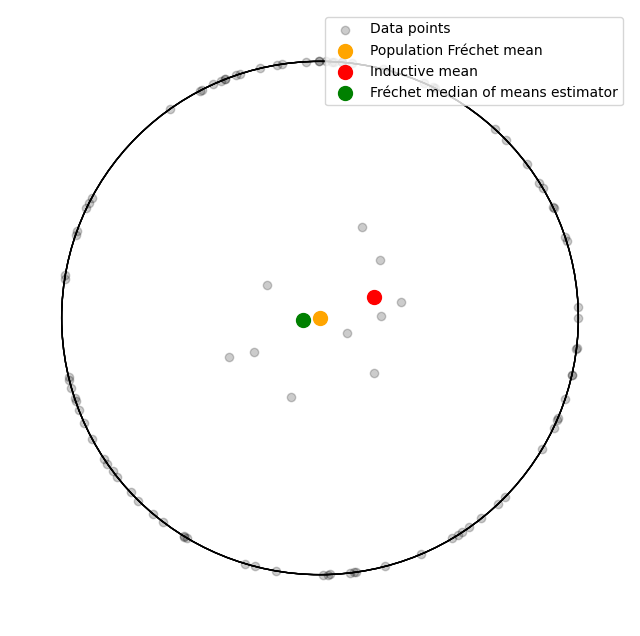}\hfill
\includegraphics[width=0.33\textwidth]{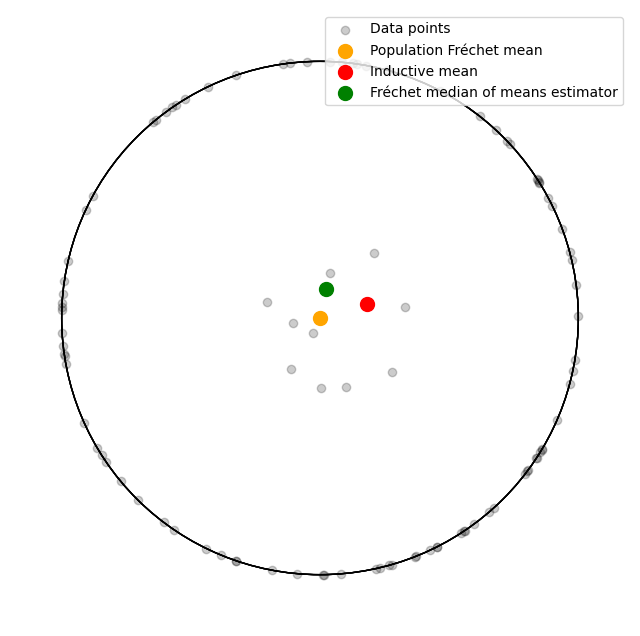}\hfill
\includegraphics[width=0.33\textwidth]{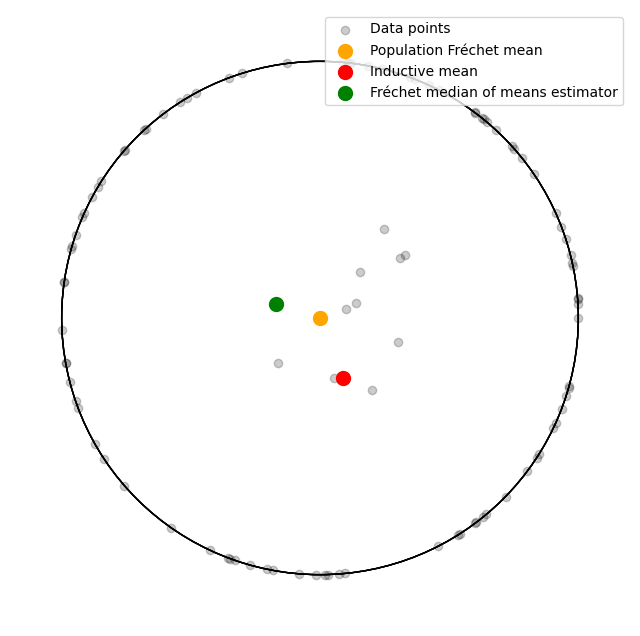}
\caption{One randomly chosen experiment results for each setting. \textbf{Rows}: $\alpha = 0, 0.1, 0.5, 0.9$ from top to the bottom. \textbf{Columns}: $k = 50, 10, 5$ from left to right. Grey points denote the samples, and the yellow, red, and green point denote the population Fr\'echet mean, inductive mean, and FMoM estimator respectively.}\label{figure_poincare_ball_res}
\end{figure*}

The results are consistent with previous experiments. When the tail is light ($\alpha = 0$), the inductive mean estimator itself achieves high accuracy and a small confidence region around 0. However, as the tail becomes heavy, our proposed estimator with the optimal block size (in this example, $k = 50$ for the most cases) performs significantly better.

\vfill

\end{document}